\documentclass{article}
\usepackage{amsfonts}
\usepackage{amsmath,amsthm,amssymb}
\usepackage[all,dvips]{xy}
\usepackage{eucal}
\usepackage{xcolor}
\usepackage{comment}
\usepackage{tikz-cd}
\usepackage[margin=2cm]{geometry}
\newcommand{\rk}{\operatorname{rk}}

\renewcommand{\deg}{\operatorname{deg}}

\renewcommand{\min}{\operatorname{min}}
\renewcommand{\max}{\operatorname{max}}

\newcommand{\Z}{\mathbb{Z}}
\newcommand{\Q}{\mathbb{Q}}
\newcommand{\C}{\mathbb{C}}
\newtheorem{theorem}{Theorem}[section]
\newtheorem{lemma}[theorem]{Lemma}
\newtheorem{corollary}[theorem]{Corollary}

\newtheorem{proposition}[theorem]{Proposition}
\newtheorem*{theorem*}{Theorem}
\theoremstyle{definition}
\newtheorem{defn}[theorem]{Definition}

\newtheorem{remark}[theorem]{Remark}

\newtheorem{example}[theorem]{Example}

\numberwithin{equation}{section}

\author{Houari Benammar Ammar}
\AtEndDocument{\bigskip{\footnotesize%
  \textsc{D\'{e}partement de math\'{e}matiques, Universit\'{e} du Qu\'{e}bec \`a Montr\'{e}al, Montr\'{e}al, QC, H2X 3Y7, Canada} \par  
  \textit{E-mail address :} \texttt{benammar\_ammar.houari@courrier.uqam.ca} 
}}
\large
\title{Slope inequality for an arbitrary divisor}
\begin{document}
\normalsize
\maketitle
\begin{abstract}
Let $f: S \longrightarrow C$ be a surjective morphism with connected fibers from a smooth complex projective surface $S$ to a smooth complex projective curve $C$ with general fiber $F$. In this paper, we develop a more general version of the slope inequality for data $(D, \mathcal{F})$, where $D$ is an arbitrary relatively effective divisor on $S$ and $\mathcal{F}$ is a locally free sub-sheaf of $f_{*}\mathcal{O}_{S}(D)$. We see how the speciality of $D$, restricted to the general fiber, plays a role in the results. Moreover, we compute some natural examples and provide  applications.
\end{abstract}
\section{Introduction}
Let $f: S \to C$ be a surjective morphism from a smooth complex projective
surface $S$ to a smooth complex projective curve $C$ with connected fibers. We call the morphism $f$
a fibration or a fibred surface. Let $D$ be a relatively effective divisor on $S$. We consider the sheaf $\mathcal{E} =
f_{*}\mathcal{O}_{S}(D)$, which is torsion free because $C$ is a curve.
Since a torsion free sheaf on curve is always a locally free sheaf, $\mathcal{E}$ is
    locally free and its rank is $h^{0}(F, {D}_{|_{F}})$, where $F$ is a general fiber of $f$ of
genus $g (F) = g$.
\par The fibration $f$ is called smooth if all its fibers are smooth, isotrivial if all
its smooth fibers are mutually isomorphic, and locally trivial if it is both
smooth and
isotrivial.
Let $\omega_{S}$ (respectively $K_{S}$) be the canonical sheaf (respectively a canonical
divisor) of $S$, $\omega_{S/C} = \omega_{S} \otimes f^{*} \omega_{C}^{\vee}$ (respectively
$K_{S/C} = K_{S} - f^{*}K_{C})$ be the relative canonical sheaf (respectively a relative
canonical divisor), where $\omega_{C}$ (respectively $K_{C}$) is the canonical sheaf of
$C$ (respectively a canonical divisor).
In particular, if $D = K_{S/C}$, then $\mathcal{E}$ is a nef vector bundle, \cite{Fujita1}, its rank is $g$, and its degree is
\[\deg(\mathcal{E}):= \deg (f_{*}\omega_{S/C})= \chi(\mathcal{O}_{S}) - \chi(\mathcal{O}_{F}).\chi(\mathcal{O}_{C})\]
\[= \chi(\mathcal{O}_{S}) - (g - 1)(b - 1) \text{,}\]
for $b := g(C) 
$. By the Leray spectral sequence, we note that
$$h^{0}(C, (f_{*}\omega_{S/C})^{\vee}) = h^{0}(C,\mathcal{R}^{1}f_{*}\mathcal{O}_{S}) = q(S) - b,$$
where $q(S) := h^{1}(S, \mathcal{O}_{S})$ is the irregularity of the surface $S$.
\par In \cite{xiaogang}, Xiao wrote a fundamental paper on fibred surfaces over curves.
He discussed the geometry of fibrations where $f$ is relatively minimal and
$g(F) \geq 2$. He proved that if $f$ is relatively minimal and not locally trivial, that is $\deg f_{*}\omega_{S/C} \neq 0$, then
\[
K_{S/C}^{2} \geq 4 \frac{g - 1}{g}\deg f_{*}\omega_{S/C} \text{.}
\]
This last result was a key ingredient in the work of Pardini \cite{Pardini} to prove Severi's conjecture \cite{Severri}, \cite{Reid}.
\par We call the above inequality a slope inequality for the relative canonical divisor. Recall that $K_{S/C}$ is a nef divisor \cite[Theorem 1.4]{ohno}.
\par Independently, Cornalba and Harris \cite{Harris} proved the above inequality for
semi-stable fibrations (that is, fibrations where all the fibers are semi-stable curves in the sense of Deligne and Mumford). Later Stoppino, \cite{stoppinoGIT}, showed that a generalization of the Cornalba-Harris approach gives a full proof of the slope inequality in which all fibrations are treated by the same method. Also, we recall that Yuan and Zhang, \cite{zhang}, gave a new approach to prove the slope inequality by giving a sense to the relative Noether inequality using Frobenius iteration techniques. Moreover, there has been interest in
giving a bound related to other geometric invariants such as the relative irregularity and the unitary rank of the fibred surface $f: S \to C$. These points have been discussed in several papers, for instance in  \cite{Barjastoppinostability}, \cite{Luandzuo} and \cite{stoppino}. Konno, in \cite{Konno}, described  $K_{S/C}^{2}$ as a sum of two parts
under some  conditions on the fibration $f$. More precisely, the first part is related to $\deg f_{*}\omega _{S/C}$ and the second one is described by the
Horikawa index
\cite{Horikawa}. Since we apply Fujita's decompositions in this paper (see Example \ref{exampleusingfujitadecomposition})
we recall them.
\begin{theorem}[First Fujita decomposition for fibred  surfaces \cite{Fujita1}]\label{First:Fujita:decomp:thm} 
\par Let $f : S \to C$ be a fibration from a smooth complex projective surface $S$ to a
smooth projective curve $C$. Then
\[
f_{*}\omega_{S/C} = \mathcal{O}_{C}^{q(S) - b} \oplus \mathcal{N} \text{,}
\]
where $\mathcal{N}$ is a nef vector sub-bundle and $h^{0}(C, \mathcal{N}^{\vee}) =
0$.
\end{theorem}
We remark that in the conclusion of Theorem \ref{First:Fujita:decomp:thm}, the trivial part comes from the nonzero global sections of the dual of $f_{*}\omega_{S/C}$. In other words, it comes from $H^{0}(C, \mathcal{R}^{1}f_{*}\mathcal{O}_{S})$.
\begin{theorem}[Second Fujita decomposition for fibred surface, \cite{Fujita2}, \cite{Cat1}, \cite{Cat2}, \cite{Cat3}]\label{Second:Fujita:decomp:thm}
Let $f : S \to C$ be a fibration as above. Then
\[
f_{*}\omega_{S/C} = \mathcal{A} \oplus \mathcal{U}\text{,} 
\]
where $\mathcal{A}$ is an ample vector sub-bundle and $\mathcal{U}$ is a unitary flat vector sub-bundle.
\end{theorem}

In the situation of Theorem \ref{Second:Fujita:decomp:thm}, we denote by $u_{f}$ the rank of $\mathcal{U}$, and call it the \emph{unitary rank} of
the fibred surface $f : S \to C$. A proof of the second Fujita decomposition is given by Catanese and Dettweiler,
\cite{Cat1}, \cite{Cat2}, \cite{Cat3}. In Theorem \ref{firstfujita} and Theorem \ref{secondfujita}, we discuss the first and the second Fujita decompositions for the adjoint cases.

\par In this paper, let $f: S \to C$ be a fibred surface, and let $D$ be a relatively effective divisor on $S$ with a general fiber $F$. Consider the data $(D, \mathcal{F})$ where $\mathcal{F} \subseteq f_{*}\mathcal{O}_{S}(D)$. The goal  is to give a lower bound for $D^{2}$ in terms of $\deg \mathcal{F}$, even if $D$ is not positive, for instance not nef. We explain how the negative part of $D$ appears in the lower bound of $D^{2}$. The method used is due to Xiao \cite{xiaogang}.
\par Now, we describe briefly the process of the main results. We start by defining the Miyaoka divisors $(N_{i})_{1 \leq i \leq k}$ of a data $(D,\mathcal{F})$, where $k$ is the length of the Harder-Narasimhan filtration $(\mathcal{F}_{i})_{1 \leq i \leq k}$ of $\mathcal{F}$ (Definition \ref{miyaoka}). We realize that this sequence of divisors can be divided into two sub-sequences. \par The first is a sequence of special Miyaoka divisors, and the second is a sequence of nonspecial Miyaoka divisors. We thus define an important number, $\Hat{n}_{(D,\mathcal{F})}$, which tells us the index of the last divisor in the first sequence (Proposition \ref{specialandnonspecial}). It is  natural  to study the sequence of rational numbers $$\left(\frac{d_{i}}{h^{0}({N_{i}}_{|_{F}}) - 1}\right)_{1\leq i \leq k}\text{,}$$ where $d_{i} = N_{i}.F$. In Theorem \ref{3.10}, we  give a uniform lower bound for this sequence. 
\par Clearly, by Clifford's Theorem, the number $2$ is a lower bound of the sub-sequence $$\left(\frac{d_{i}}{h^{0}({N_{i}}_{|_{F}}) - 1}\right)_{1 \leq i \leq \Hat{n}_{(D,\mathcal{F})}}\text{,}$$ and we prove that the number 
$$\beta_{D}:= 1 + \frac{g(F)}{h^{0}(F, {D}_{|_{F}}) - 1}$$ is a lower bound of the sub-sequence 
$$\left(\frac{d_{i}}{h^{0}({N_{i}}_{|_{F}}) - 1}\right)_{\Hat{n}_{(D,\mathcal{F})} + 1 \leq i \leq k}
$$ for a fixed $D$, and any $\mathcal{F} \subseteq f_{*}\mathcal{O}_{S}(D).$ Moreover,  we will see that this sub-sequence decreases. Following the Modified Xiao Lemma (Lemma \ref{4.1}), we observe that we need to find a constant $\alpha_{(D,\mathcal{F})}$ for the data $(D, \mathcal{F})$ such that $$d_{i} \geq \alpha_{(D, \mathcal{F})}(h^{0}(F, {N_{i}}_{|_{F}}) - 1) \geq \alpha_{(D, \mathcal{F})} (r_{i} -1)$$ (here $r_{i}:= \rk \mathcal{F}_{i}$) with the aim that $\deg \mathcal{F}$ appears. Thus, naturally  Definition 
\ref{alphadefinition} follows.
The main result of the paper is the following theorem.
\begin{theorem*}[Theorem \ref{main result}]
Let $f: S \to C$ be a fibred surface. Consider the data $(D, \mathcal{F})$, where $D$ is  a relatively effective divisor and $\mathcal{F} \subseteq f_{*}\mathcal{O}_{S}(D)$ is a locally free sub-sheaf on $C$. Assume that $\mathcal{F}$ is not semi-stable. Here, we set:
\[ t_{(D, \mathcal{F})}:= \left\{
    \begin{array}{ll}
\max\{i| \mu_{i} \geq 0\}  & \mbox{if $\mu_{1} \geq 0$}  
\\ -\infty & \mbox{otherwise} \text{.}

\end{array}
\right. 
\]  
\begin{itemize}
  \item[(1).] If $t_{(D,\mathcal{F})} = 1$, then
  \[
  D^{2} \geq \frac{2d_{1}}{r_{1}} \deg \mathcal{F}_{1} + 2 \epsilon^{*}D.Z_{1} - Z_{1}^{2} \geq \frac{2d_{1}}{r_{1}} \deg \mathcal{F} + 2 \epsilon^{*}D.Z_{1} - Z_{1}^{2} \text{.} 
  \]  
    \item[(2).] If $1 < t_{(D,\mathcal{F})} \leq k$, then
  \[
  D^{2} \geq \frac{2\alpha_{(D,\mathcal{F}_{t_{(D,\mathcal{F})}})} d_{t_{(D, \mathcal{F})}}}{d_{t_{(D,\mathcal{F})}} + \alpha_{(D,\mathcal{F}_{t_{(D, \mathcal{F})}})}} \deg \mathcal{F}_{t_{(D,\mathcal{F})}} + 2 \epsilon^{*}D.Z_{t_{(D,\mathcal{F})}} - Z_{t_{(D,\mathcal{F})}}^{2}
  \]
  \[\geq \frac{2\alpha_{(D,\mathcal{F}_{t_{(D,\mathcal{F})}})} d_{t_{(D,\mathcal{F})}}}{d_{t_{(D,\mathcal{F})}} + \alpha_{(D,\mathcal{F}_{t_{(D,\mathcal{F})}})}} \deg \mathcal{F} + 2 \epsilon^{*}D.Z_{t_{(D,\mathcal{F})}} - Z_{t_{(D,\mathcal{F})}}^{2}\text{.} 
  \]  
    \item[(3).] If $t_{(D,\mathcal{F})} = -\infty$, set:
  \[ C_{(D,\mathcal{F})} := \left\{
    \begin{array}{ll}
\frac{2\alpha_{(D,\mathcal{F})} d_{k}}{-\alpha_{(D,\mathcal{F})} + 2\alpha_{(D,\mathcal{F})} r_{k} - d_{k}}  & \mbox{if $\alpha_{(D,\mathcal{F})} - 2\alpha_{(D,\mathcal{F})} r_{k} + 2d_{k} \leq 0$}  
\\ 3\alpha_{(D,\mathcal{F})} + 2d_{k} - 2\alpha_{(D,\mathcal{F})} r_{k} & \mbox{otherwise} \text{.}

\end{array}
\right. 
\]  
Then
    \[
D^{2} \geq C_{(D,\mathcal{F})}. \deg \mathcal{F} + 2 \epsilon^{*}D.Z_{k} - Z_{k}^{2} \text{.} 
    \]
\end{itemize}
In addition, if $d_{k} = \alpha_{(D,\mathcal{F})} (r_{k} - 1)$, then we have the following inequality which is independent of $t_{(D,\mathcal{F})}$:
\[
D^{2} \geq \frac{2d_{k}}{r_{k}} \deg \mathcal{F} + 2 \epsilon^{*}D.Z_{k} - Z_{k}^{2} \text{.} 
\]    
\end{theorem*}
\par We make some remarks about the notations introduced in the previous theorem. 
\par There is an iterated blowing-up morphism $\epsilon: \Hat{S} \to S $ that is constructed in Proposition \ref{proposition3.3}.   By definition, the divisor $Z_{i}$ on $\Hat{S}$ is a \emph{fixed part} of $\mathcal{F}_{i},\forall i; 1 \leq i \leq k$ (Definition \ref{miyaoka}). The number $\mu_{i}: = \mu(\mathcal{F}_{i}/\mathcal{F}_{i-1})$ is the slope of the quotient $\mathcal{F}_{i}/\mathcal{F}_{i-1}$ (Proposition \ref{hardernarasimhanfiltration}) of the Harder-Narasimhan filtration of $\mathcal{F}$.

\par We interpret the result of Theorem \ref{main result} as follows. If we consider the data $(D, \mathcal{F})$, we estimate the lower bound of $D^{2}$ by a sum of two parts. The  first one is related to the degree of the sub-bundle $\mathcal{F}_{t_{(D,\mathcal{F})}} \subseteq \mathcal{F}$ and the geometry of the Miyaoka divisors $(N_{i})_{1\leq i \leq t_{(D,\mathcal{F})}}$. The second part captures the negativity of $D$ when $D$ is not positive. The case when $\mathcal{F}$ is semi-stable is discussed in Remark \ref{remark4.6}. As an application, we apply the previous theorem to 
the data $(D, f_{*}\mathcal{O}_{S}(D))$
when $D$ is relatively nef. Thus,  we obtain Corollary \ref{secondmainresult}.
\par This new result pertains to the way in which the constant $\alpha_{(D,f_{*}\mathcal{O}_{S}(D))}$,  defined in Definition \ref{alphadefinition}, arises in lower bounds for $D^2$.  Moreover, the last paragraph in  Corollary \ref{secondmainresult} extends the result in \cite[Thoerem 5]{remarkonslope}, see Example \ref{remarkslopepaper} and Theorem \ref{invarianttheorem}. The form of  Corollary \ref{comparewithstoppino}  can be compared to \cite[Theorem 3.20]{stoppino}. 
\par Among other results, our contribution in this article is to give an explicit description  of the constant  $\alpha$ that arises there.  This is achieved by the constant $\alpha_{(D,\mathcal{F})}$ (which is defined in Definition \ref{alphadefinition}).

\par Recently, in \cite[Theorem E]{godogniviviani}, the authors proved a slope inequality for a data $(D, f_{*}\mathcal{O}_{X}(D))$ where $D$ is a relatively effective  divisor on an $n$-dimensional  fibred variety $f \colon X \rightarrow C$ over a curve with $n \geq 2$, and under the assumptions that $D$ and $f_{*}\mathcal{O}_{X}(D)$  are both  nef. If we apply their result for fibred surfaces, we  can see that our result is  sharper. More precisely, in \cite[Theorem E]{godogniviviani},  they proved if $D_{|_{F}}$ is nef and big, and nonspecial, then 
\[
D^{2} \geq  2 \frac{D.F}{D.F + 1} \deg f_{*}\mathcal{O}_{S}(D) \text{.}
\]
However, if we apply  Corollary \ref{comparewithstoppino} for the data $(D,f_{*}\mathcal{O}_{S}(D))$ where $D$ and $f_{*}\mathcal{O}_{S}(D)$  are both nef, and $D_{|_{F}}$ is nonspecial, then we have 
\[
D^{2} \geq \frac{2\alpha_{(D,f_{*}\mathcal{O}_{S}(D))}D.F }{D.F + \alpha_{(D, f_{*}\mathcal{O}_{S}(D))}} \deg f_{*}\mathcal{O}_{S}(D) \text{.}
\]
Since $\alpha_{(D,f_{*}\mathcal{O}_{S}(D))} > 1$  for $g(F) \neq 0$, we see that the constant $\frac{\alpha_{(D,f_{*}\mathcal{O}_{S}(D))}D.F }{D.F + \alpha_{(D, f_{*}\mathcal{O}_{S}(D))}}$ is strictly bigger than $\frac{D.F}{D.F + 1}$ in general.   

\par
\textbf{Acknowledgements.}  I am very grateful to  my advisors Steven Lu and Nathan Grieve for their constant help, invaluable advice and financial support. I would like to thank the anonymous referee and Jungkai Chen for their very careful reading of this work and for providing numerous comments and helpful suggestions.

\par In the next sections,  $f : S \to C$ is a fibred surface from a smooth complex projective surface
$S$ to a smooth complex projective curve $C$, and $D$ is a relatively effective divisor on $S$. Also, throughout this article unless stated otherwise all locally free sheaves are assumed to be nonzero.

\section{Rational map to a projective bundle}

Let $f: S \to C$ be a fibred surface and $D$ be a relatively effective divisor on $S$. Let $\mathcal{F} \subseteq f_{*}\mathcal{O}_{S}(D)$
be a locally free sub-sheaf of rank $r_{\mathcal{F}}$. There exist always the following
commutative diagrams:
\[
\begin{tikzcd}
S \arrow[r,dashrightarrow, "\psi"]
\arrow{rd}[swap]{f} & \mathbb{P}_{C}(f_{*}\mathcal{O}_{S}(D)) \arrow[d,"\pi"]\\
& C
\end{tikzcd}
\]
\[
\begin{tikzcd}
S \arrow[r,dashrightarrow, "\psi_{\mathcal{F}}"]
\arrow{rd}[swap]{f} & \mathbb{P}_{C}(\mathcal{F}) \arrow[d,"\pi_{\mathcal{F}}"]\\
& C
\end{tikzcd}
\]

\par In the above, $\mathbb{P}_{C}(f_{*}\mathcal{O}_{S}(D))$ (respectively $\mathbb{P}_{C}(\mathcal{F})
$) is the projective bundle of one
dimensional quotients (Grothendieck’s notations) of $f_{*}\mathcal{O}_{S}(D)$ (respectively
of $\mathcal{F}$) and $\pi$ (respectively $\pi_{\mathcal{F}}$) is the projective
morphism from $\mathbb{P}_{C}(f_{*}\mathcal{O}_{S}(D))$ to $C$ (respectively from $\mathbb{P}_{C}(\mathcal{F})$ to $C$). The maps $\psi$ and $\psi_{\mathcal{F}}$ are
rational and defined by the following evaluation maps:
\[
f^{*}f_{*}\mathcal{O}_{S}(D) \longrightarrow \mathcal{O}_{S}(D) \text{,}
\] 
(respectively)
\[
f^{*}\mathcal{F} \longrightarrow \mathcal{O}_{S}(D) \text{.}
\] 
\begin{remark}
\begin{itemize}
\item If we assume $D$ is $f$-globally generated in the sense that
\[
f^{*}f_{*}\mathcal{O}_{S}(D) \longrightarrow \mathcal{O}_{S}(D)
\]
is surjective, then $\psi$ is a morphism by \cite[Proposition II.7.12]{Hartshornebook}.
\item If the map\[
f^{*}\mathcal{F} \longrightarrow \mathcal{O}_{S}(D)
\] is surjective, then $\psi_{\mathcal{F}}$ is also  a morphism by \cite[Proposition II.7.12]{Hartshornebook}.
\end{itemize}
\end{remark}
\par Take a sufficiently very ample divisor $A$ on $C$ such that $f_{*}\mathcal{O}_{S}(D) \otimes \mathcal{O}(A)$ is a very ample vector bundle.  Then the rank of $f_{*}
\mathcal{O}_{S}(D) \otimes \mathcal{O}(A)$ is
$$r = H^{0}(F, D_{|_{F}}) \text{,}$$
and
$$\deg (f_{*}\mathcal{O}_{S}(D) \otimes \mathcal{O}(A)) = \deg f_{*}\mathcal{O}_{S}
(D) + r\deg(A).$$
\begin{remark}
The fact that $f_{*}\mathcal{O}_{S}(D) \otimes \mathcal{O}(A)$ is globally generated is equivalent to $\exists \hspace{0.1cm} n > 0$ such that we have a surjective map:
\[
\mathcal{O}_{C}^{\oplus n} \longrightarrow f_{*}\mathcal{O}_{S}(D) \otimes 
\mathcal{O}(A) \text{,}
\]
and to the fact that, for any $y \in C$ the map given by evaluation of
sections is surjective:
\[
H^{0}(C, f_{*}\mathcal{O}_{S}(D) \otimes \mathcal{O}(A)) \longrightarrow f_{*}\mathcal{O}_{S}(D) \otimes \mathcal{O}(A) _{|_{y}}\text{.}
\]
We thus remark that each section in $H^{0}(F, D_{|_{F}})$ comes from some section of
$H^{0}(C, f_{*}\mathcal{O}_{S}(D) \otimes \mathcal{O}(A))$. In other words, the
following map is surjective:
\[
H^{0}(C, f_{*}\mathcal{O}_{S}(D) \otimes \mathcal{O}(A)) \simeq H^{0}(S, D +
f^{*}A) \longrightarrow H^{0}(F, D_{|_{F}}) \text{.}
\]
\end{remark}
\par Now, 
$\mathbb{P}_{C}(f_{*}\mathcal{O}_{S}(D) \otimes \mathcal{O}(A))$
and 
$\mathbb{P}_{C}(f_{*}\mathcal{O}_{S}(D))$ are isomorphic by an isomorphism $s$, (\cite[Lemma 7.9]{Hartshornebook}).  
\par The rational map $$\phi : S \dashrightarrow \mathbb{P}_{C}
(f_{*}\mathcal{O}_{S}(D) \otimes \mathcal{O}(A)) \text{,}$$ defined by $$f^{*}(f_{*}\mathcal{O}_{S}(D) \otimes \mathcal{O}(A)) \longrightarrow \mathcal{O}_{S}(D) \otimes f^{*} \mathcal{O}(A) $$ is the rational map given by the linear system $|D +
f^{*}A|$, and restricted to the general fiber $F$ of $f$, $\phi_{|_{F}}$ is the
map defined by $|D_{|_{F}}|$.
\par The line bundle $\mathcal{O}_{\mathbb{P}_{C}(f_{*}\mathcal{O}_{S}(D) \otimes \mathcal{O}(A))}(1)$ on $\mathbb{P}_{C}(f_{*}\mathcal{O}_{S}(D) \otimes \mathcal{O}
(A))$ is very ample. Then it gives an embedding of this last projective bundle to a
projective space $\C\mathbb{P}^{N}$ for some $N > 0$. We thus have the following commutative diagram:
\[
\begin{tikzcd}
S
\arrow[r,dashrightarrow, "\phi"]
\arrow{rd}[swap]{f} & \mathbb{P}_{C}(f_{*}\mathcal{O}_{S}(D) \otimes \mathcal{O}(A)) \subseteq \C\mathbb{P}^{N} \arrow[d, "\pi_{A}"] \\
& C
\end{tikzcd}
\]
where $\pi_{A}$ is the projection map, again we have $\psi = s \circ \phi$, the
rational map $\phi$ is defined by the complete linear system $|D + f^{*}A|$. If it
has no nontrivial fixed part, then its image is contained in any hyperplane.
\par We assume that there is a fixed part $Z$ of $|D + f^{*}A|$. Thus, the linear system
$|D - Z + f^{*}A|$ factorizes the map $\phi$ defined by $|D + f^{*}A|$ and  the following properties are satisfied:
\begin{itemize}
\item[(1).] The fixed part $Z$ of $|D + f^{*}A|$, restricted to $F$, is just the fixed
part of the complete linear system $|D_{|_{F}}|$.
\item[(2).] The system $|D - Z + f^{*}A|$ has no fixed part, so it has only a finite
number of base points.
\item[(3).] The fixed part $Z$ is a divisor such that the morphism: $$f^{*}(f_{*}\mathcal{O}_{S}(D) \otimes \mathcal{O}(A)) \longrightarrow \mathcal{O}_{S}(D - Z) \otimes f^{*} \mathcal{O}(A)\text{,}$$ is surjective in codimension 1.
\item[(4).] We can assume that $Z$ has no
horizontal components because $A$ is sufficiently ample.
\end{itemize}
\begin{theorem}\label{2.3}
There exist a series of blow ups  $\epsilon : \Tilde{S} \to S$  and a morphism
\[\lambda_{A}: \Tilde{S} \longrightarrow \mathbb{P}_{C}(f_{*}\mathcal{O}_{S}(D) \otimes \mathcal{O}(A))  \]
such that the following diagram is commutative:
\[
\begin{tikzcd}
\Tilde{S} \arrow{d}[swap]{\epsilon} \arrow[rd, bend left, "\lambda_{A}"]
\\
S
\arrow[r,dashrightarrow, "\phi"]
\arrow[dr,"f", right ]
& \mathbb{P}_{C}(f_{*}\mathcal{O}_{S}(D) \otimes \mathcal{O}(A))   \subseteq \C\mathbb{P}^{N} \arrow[d, "\pi_{A}"] \\
&  C
\end{tikzcd}
\]
$\phi \circ \epsilon = \lambda_{A}$ and
\[
(\lambda_{A})^{*}\mathcal{O}_{\mathbb{P}_{C}(f_{*}\mathcal{O}_{S}(D) \otimes \mathcal{O}(A)) }(1) = \epsilon^{*}(\mathcal{O}(D -
Z) \otimes f^{*}\mathcal{O}(A)) \otimes \mathcal{O}(-E) \text{,}
\]
where $E$ is the exceptional divisor of $\epsilon$.
\end{theorem}
\begin{remark}
$\epsilon^{*}(\mathcal{O}(D - Z) \otimes f^{*}\mathcal{O}(A)) \otimes \mathcal{O}
(-E)$ is globally generated.
\end{remark}
\begin{proof}[Proof of Theorem \ref{2.3}] The linear system  $|D - Z + f^{*}A|$ has at worst finitely many base  points. If there are none, $\phi$ is a morphism and there is
nothing to prove. We suppose that there is a base point $x$ in $|D - Z + f^{*}A|$.
We take the blow-up in $x$ defined by $\epsilon^{1}$, so  $ |(\epsilon^{1})^{*}
(D - Z + f^{*}A)|$ has a fixed part $k_{1}E_{1}$ with $k_{1} \in \Z, k_{1}\geq1$
and $ |D_{1}| =|(\epsilon^{1})^{*}(D - Z + f^{*}A) -k_{1}E_{1}|$ has no fixed part.
Hence, it defines
a rational map, $\lambda^{1} : S_{1} \dashrightarrow \mathbb{P}_{C}(f_{*}\mathcal{O}_{S}(D) \otimes \mathcal{O}(A))  $ which is identical to
$\phi \circ \epsilon^{1}$. If $\lambda^{1}$
is a morphism, then we are done; if not, we repeat the process. Thus, we get
by induction a sequence $\epsilon^{i}: S_{i} \longrightarrow S_{i-1}$ of blow-ups
and a linear system $|D_{i}|$ with no fixed part, where $D_{i} = (\epsilon^{i})^{*}
D_{i-1} - k_{i}E_{i} $ for $i \geq 1$. 
\par In other words,
we arrive at a system $D_{n}$ with no base points, which defines
a morphism:
\[
\epsilon = \epsilon^{1} \circ ... \circ \epsilon^{n} : \Tilde{S} \longrightarrow S \text{.}
\]
We conclude that $|\epsilon^{*}(D - Z + f^{*}A) -E|$ defines a morphism $$\Tilde{S} \overset{\lambda_{A}}{\longrightarrow} \mathbb{P}_{C}(f_{*}\mathcal{O}_{S}(D) \otimes \mathcal{O}(A)) $$ such that
\[
\epsilon^{*}(\mathcal{O}(D - Z) \otimes f^{*}\mathcal{O}(A)) \otimes \mathcal{O}(-
        E) = (\lambda_{A})^{*}\mathcal{O}_{\mathbb{P}_{C}(f_{*}\mathcal{O}_{S}(D) \otimes \mathcal{O}(A))}(1)
\]
where $E = \sum_{i=1}^{i=n} k_{i}E_{i}$ is the exceptional divisor.
\end{proof}
\par The last proof is inspired by the proof of \cite[Theorem 2.7]{beauville}.
\begin{corollary}\label{2.5}
There exists a morphism $\lambda : \Tilde{S} \to \mathbb{P}_{C}(f_{*}\mathcal{O}_{S}(D)) $ such that the following diagram is commutative:
\[
\begin{tikzcd}
\Tilde{S} \arrow[dr, bend left, "\lambda"]
\arrow{d}[swap]{\epsilon}
& &\\
S \arrow[r,dashrightarrow, "\psi"]
\arrow[rd,"f"] & \mathbb{P}_{C}(f_{*}\mathcal{O}_{S}(D)) \arrow[d,"\pi"]\\
& C
\end{tikzcd}
\]
Moreover, we have
\[\lambda^{*}(\mathcal{O}_{\mathbb{P}_{C}(f_{*}\mathcal{O}_{S}(D))}(1)) = \epsilon^{*}(\mathcal{O}(D - Z)) \otimes \mathcal{O}(-E)\text{.}\]
\end{corollary}
\begin{proof} By Theorem \ref{2.3},  there exists
a morphism $$\lambda_{A} : \Tilde{S} \longrightarrow \mathbb{P}_{C}(f_{*}\mathcal{O}_{S}(D) \otimes \mathcal{O}(A))$$ which has the following property:
\[
(\lambda_{A})^{*}\mathcal{O}_{\mathbb{P}_{C}(f_{*}\mathcal{O}_{S}(D) \otimes \mathcal{O}(A))}(1) = \epsilon^{*}(\mathcal{O}(D - Z) \otimes f^{*}\mathcal{O}
(A)) \otimes \mathcal{O}(-E)
.\]
But there exists an isomorphism $$s : \mathbb{P}_{C}(f_{*}\mathcal{O}_{S}(D) \otimes \mathcal{O}(A)) \longrightarrow \mathbb{P}_{C}(f_{*}\mathcal{O}_{S}(D)) $$ 
such that
\[
\mathcal{O}_{\mathbb{P}_{C}(f_{*}\mathcal{O}_{S}(D) \otimes \mathcal{O}(A))}(1) =
s^{*}\mathcal{O}_{\mathbb{P}_{C}(f_{*}\mathcal{O}_{S}(D))}(1) \otimes \pi_{A}^{*}\mathcal{O}(A)\text{.}
\]
Therefore, 
\[
(s \circ \lambda_{A})^{*}\mathcal{O}_{\mathbb{P}_{C}(f_{*}\mathcal{O}_{S}(D))}(1) \otimes (\pi_{A} \circ \lambda_{A})^{*}\mathcal{O}(A) = \epsilon^{*}(\mathcal{O}(D - Z)) \otimes (f \circ \epsilon)^{*}\mathcal{O}
(A) \otimes \mathcal{O}(-E)\]
implies
\[
(s \circ \lambda_{A})^{*}\mathcal{O}_{\mathbb{P}_{C}(f_{*}\mathcal{O}_{S}(D))}(1) = \epsilon^{*}(\mathcal{O}(D - Z)) \otimes \mathcal{O}(-E) \text{.}
\]
We take $$\lambda = s \circ \lambda_{A} \text{.}$$
\end{proof}
\begin{remark}\label{2.6}
More generally, for $\mathcal{F} \subseteq f_{*}\mathcal{O}_{S}(D)$ a locally free sub-sheaf,
we take  a sufficiently very ample divisor $A$ on $C$ such that $\mathcal{F} \otimes \mathcal{O}(A)$ is very ample. Let $L_{\mathcal{F}}$ be the linear sub-system of $|D +
f^{*}A|$ which corresponds to the sections of $H^{0}(\mathcal{F} \otimes \mathcal{O}
(A))$.
Let $Z_{\mathcal{F}}$ be the fixed part of $L_{\mathcal{F}}$, so $L_{\mathcal{F}} -
Z_{\mathcal{F}}$ has no fixed part and it corresponds to a rational map from $S$
to a projective sub-variety of $\mathbb{P}_{C}(\mathcal{F} \otimes \mathcal{O}(A))$. 
By the same arguments as above, $\exists \hspace{0.1cm} \Tilde{S}_{\mathcal{F}} \overset{\epsilon_{\mathcal{F}}}{\longrightarrow} S$ which is a chain of blow ups and $\exists \lambda_{\mathcal{F}} : \Tilde{S}_{\mathcal{F}} \longrightarrow \mathbb{P}_{C}(\mathcal{F})$ such that the following diagram is commutative:
\[
\begin{tikzcd}
\Tilde{S}_{\mathcal{F}} \arrow[dr, bend left, "\lambda_{\mathcal{F}}"]
\arrow{d}{\epsilon_{\mathcal{F}}}
& &\\
S \arrow[r,dashrightarrow, "\psi_{\mathcal{F}}"]
\arrow[rd,"f"] & \mathbb{P}_{C}(\mathcal{F}) \arrow[d,"\pi_{\mathcal{F}}"]\\
& C
\end{tikzcd}
\]
and \[
(\lambda_{\mathcal{F}})^{*}(\mathcal{O}_{\mathbb{P}_{C}(\mathcal{F})}(1)) = \epsilon_{\mathcal{F}}^{*}(\mathcal{O}(D - Z_{\mathcal{F}})) \otimes \mathcal{O}(-
E_{\mathcal{F}})
\text{,}\]
where $E_{\mathcal{F}}$ is the exceptional divisor of $\epsilon_{\mathcal{F}}$.
\end{remark}

\section{Harder-Narasimhan filtration}

\par In this section, we study the Harder-Narasimhan filtration within the context of fibred
surfaces.
\begin{proposition}[{\cite{HarderNarismhan}}]\label{hardernarasimhanfiltration}
Let $\mathcal{F}$ be a vector bundle over a smooth projective curve
$C$. There exists a unique sequence of vector sub-bundles of $\mathcal{F}$:
\[
0 = \mathcal{F}_{0} \subsetneq \mathcal{F}_{1} \subsetneq ... \subsetneq \mathcal{F}_{k-1} \subsetneq \mathcal{F}_{k} = \mathcal{F}\text{,}
\]
that satisfies the following conditions:
\begin{itemize}
\item[(1).] For $i = 1, . . . , k$, $\mathcal{F}_{i}/\mathcal{F}_{i - 1}$ is a
semi-stable vector bundle.
\item[(2).] For any $i = 1 ,..., k$, setting $\mu_{i} : = \mu (\mathcal{F}_{i}/\mathcal{F}_{i - 1}) = \frac{\deg(\mathcal{F}_{i}/\mathcal{F}_{i - 1})}{\rk(\mathcal{F}_{i}/\mathcal{F}_{i - 1})}$, we have:
\[
\mu_{1} > \mu_{2} > ... > \mu_{k}\text{.}
\]
\end{itemize}
\end{proposition}
\par In the context of Proposition \ref{hardernarasimhanfiltration} above, the filtration is called the \emph{Harder-Narasimhan filtration} of $\mathcal{F}$.
We set $\mu_{f} = \mu_{k}$ and call it the \emph{final slope} of $\mathcal{F}$.
The following elementary lemma is important in what follows.
\begin{lemma}\label{degreelemma}
Let $r_{i}$ be the rank of $\mathcal{F}_{i}$. Then
\[
\deg \mathcal{F} = \sum_{i = 1}^{k - 1} r_{i} (\mu_{i} - \mu_{i + 1 }) + r_{k}\mu_{k} \text{.}
\]
\end{lemma}
\begin{proof} Indeed, we consider the exact sequence:
\[
0 \longrightarrow \mathcal{F}_{k-1} \longrightarrow \mathcal{F}_{k} \longrightarrow
\mathcal{F}_{k} / \mathcal{F}_{k-1} \longrightarrow 0
\text{.}\]
From the additivity of degree, we have
\[
\deg \mathcal{F}_{k} = \deg \mathcal{F}_{k-1} + \deg \mathcal{F}_{k} / \mathcal{F}_{k-1}
\text{.}\]
Similarly, we have
\[
\deg \mathcal{F}_{k-1} = \deg \mathcal{F}_{k-2} + \deg \mathcal{F}_{k-1} / \mathcal{F}_{k-2}
\text{.}\]
And so, by induction, we can conclude that
\[
\deg \mathcal{F}_{k} = \sum_{i = 1}^{k} \deg \mathcal{F}_{i} / \mathcal{F}_{i-1}
\text{.}\]
From the definition of slope, for every $i = 1, .., k$ we have: $$\deg \mathcal{F}_{i} / \mathcal{F}_{i-1} = \mu_{i}(r_{i} - r_{i-1}) \text{.}$$ Thus, we obtain the
desired formula.
\end{proof}
Consider now a fibred surface $f: S \to C$. Let $F$ be its general fiber, let $D$ be a relatively effective divisor on $S$, and let  $(\mathcal{F}_{i})_{0 \leq i \leq k}$ be the Harder-Narasimhan filtration of $ \mathcal{F} \subseteq \mathcal{E} =
f_{*}\mathcal{O}_{S}(D)$. By a  repeated application of Remark \ref{2.6} to each of the $\mathcal{F}_{i}$'s, we have the following proposition.
\begin{proposition}\label{proposition3.3}
There
exists a suitable smooth projective surface $\Hat{S}$ and a birational morphism $\epsilon: \Hat{S} \to S$ such that the following diagram is commutative $\forall i; 1 \leq i \leq k$:
\[
\begin{tikzcd}
\Hat{S} \arrow[ddr, bend left, "\lambda_{i} "]
\arrow[d, "\epsilon_{i}"]
\arrow[dd, bend right, swap, "\epsilon"]
& &\\
\Tilde{S}_{\mathcal{F}_{i}} \arrow[dr, "\lambda_{\mathcal{F}_{i}}"]
\arrow[d, "\epsilon_{\mathcal{F}_{i}}"]
& &\\
S \arrow[r,dashrightarrow, "\psi_{\mathcal{F}_{i}}"]
\arrow[rd,"f"] & \mathbb{P}_{C}(\mathcal{F}_{i}) \arrow[d,"\pi_{i}"]\\
& C
\end{tikzcd}
\]
Where $\epsilon_{\mathcal{F}_{i}}$ is a blow-up morphism along a finite number of points $\{x_{i_{1}},\dots,x_{i_{m_i}} \}$, $1 \leq i \leq k$, as defined in the proof of Theorem \ref{2.3} and Remark \ref{2.6}.
\end{proposition}
\begin{proof}
Set $\bigcup_{1\leq i \leq k}\{x_{i_{1}},\dots,x_{i_{m_i}} \} = \{q_{1},\dots, q_{m}\}$ and  let $\Hat{S} = \mathbb{BL}_{q_{1},..., q_{m}}(S)$ be the blowing up of $S$ at $q_{1},\dots,q_m$. Then there exists a blowing up morphism $\epsilon_{i}$ along $\{q_{1},..., q_{m}\} \backslash
\{x_{i_{1}},\dots,x_{i_{m_i}} \}$  which fits into the diagram above and 
$\epsilon = \epsilon_{i} \circ \epsilon_{\mathcal{F}_{i}}: \Hat{S} \to S$. 
As before, we fix a sufficiently ample divisor $A$ on $C$.
Define the map $$\lambda_{i}: \Hat{S} \to \mathbb{P}_{C}(\mathcal{F}_{i})$$ by the linear system 
$|\epsilon^{*}(D - Z_{\mathcal{F}_{i}}) - E|$, and furthermore we have
$$\lambda_{i} = \lambda_{\mathcal{F}_{i}} \circ \epsilon_{i}\text{.}$$
Moreover, for any $\mathcal{F}_{i}$ in the filtration, we have
\[
\lambda_{i}^{*}(\mathcal{O}_{\mathbb{P}_{C}(\mathcal{F}_{i})}(1)) = \epsilon^{*}(\mathcal{O}(D - Z_{\mathcal{F}_{i}})) \otimes  \mathcal{O}(-E)
\text{.}\]
Where  $Z_{\mathcal{F}_{i}}
$ is the fixed part of $L_{\mathcal{F}_{i}} \subseteq |D
+ f^{*}A|$ which correspond to sections of $H^{0}(\mathcal{F}_{i} \otimes \mathcal{O}(A))$. Here $E$ is the exceptional divisor of $\epsilon$.
\end{proof}
\par In what follows, we study the fibred surface $\Hat{f}:= f \circ \epsilon : \Hat{S} \to C$ and we denote by $F$ its general fiber.
 \begin{defn}[Compare with  {\cite[Definition 3.11]{stoppino}}]\label{miyaoka}
In this setting, just described above, $\forall i ; 1 \leq i \leq k$, we define
the following divisors on  $\Hat{S}$:
\begin{itemize}
\item $Z_{i}:= \epsilon^{*}Z_{\mathcal{F}_{i}} + E$, the \emph{fixed
part} of the vector sub-bundle $\mathcal{F}_{i}, \forall i; 1 \leq i \leq k$. 
\item $M_{i}:= \lambda_{i}^{*}(\mathcal{O}_{\mathbb{P}_{C}(\mathcal{F}_{i})}(1))$, the \emph{moving part}
of the vector sub-bundle $\mathcal{F}_{i}, \forall i; 1 \leq i \leq k$.
\item Set $N_{i}: = M_{i} - \mu_{i} F, \forall i; 1 \leq i \leq k $. We call this the $i^{th}$
\emph{Miyaoka divisor}. 
\end{itemize}
\end{defn}
\par Applying, \cite{Miyaoka}, \cite{Nakayama} and \cite[Proposition 6.4.11]
{positivityvolume2}, we prove the following Lemma \ref{Miyaoka:lemma} using the language of $\Q$-twisted vector bundles. First,  we recall the definition of a $\Q$-twisted vector bundle.
\begin{defn}({See \cite[Definition 6.2.1]{positivityvolume2}}). A \emph{$\Q$-twisted vector bundle} $\mathcal{E}\langle \delta \rangle$ on a projective variety $X$ consists of a vector bundle $\mathcal{E}$ defined up to isomorphism, and a $\Q$-numerical equivalence class $\delta \in N_{\Q}^{1}(X)$. If $D$ is a $\Q$-divisor, we write $\mathcal{E}\langle D \rangle$ for the twist $\mathcal{E}$ by the numerical equivalence class of 
$D$. We define \emph{$\Q$-isomorphism} of $\Q$-twisted bundles to be the equivalence relation generated by saying that $\mathcal{E}\langle A + D \rangle$ is equivalent to $(\mathcal{E}\otimes \mathcal{O}(A))\langle D \rangle$ for all integral divisors $A$ and $\Q$-divisors $D$.
\end{defn}
\begin{lemma}[Compare with  {\cite[Corollary IV.3.8]{Nakayama}}]\label{Miyaoka:lemma}
 $\forall i; 1 \leq i \leq k, \hspace{0.2cm}N_{i}$ are nef divisors on $\Hat{S}$.
\end{lemma}
\begin{proof}
Fix $i; 1 \leq i \leq k$, let us see that the $\Q$-twisted vector bundle $\mathcal{F}_{i}\langle - \frac{c_{1}(\mathcal{F}_{i}/\mathcal{F}_{i-1})}{\
rk(\mathcal{F}_{i}/\mathcal{F}_{k-1})} \rangle $ is nef.
\par We define the following quotient bundles and $\Q$-divisors: \[\left\{
\begin{array}{ll}
\mathcal{G}_{i} = \mathcal{F}_{i}/\mathcal{F}_{i-1} \\
\delta_{i} = \frac{c_{1}( \mathcal{G}_{i})}{\rk(\mathcal{G}_{i})}.
\end{array}
\right.\]
However, $\mathcal{G}_{i}\langle - \delta_{i} \rangle$ is a nef vector bundle by \cite[Proposition
6.4.11]{positivityvolume2}. Furthermore, $\deg \delta_{i} = \mu_{i}$. Thus
\[
-\deg \delta_{1} < -\deg \delta_{2} < ... < -\deg \delta_{k}
.\]
Using the following  exact sequence of vector bundles:
\[
0 \longrightarrow \mathcal{F}_{i-1} \longrightarrow \mathcal{F}_{i}  \longrightarrow \mathcal{G}_{i} \longrightarrow 0 \text{,}
\]
we  next prove that $\mathcal{F}_{i} \langle - \delta_{i} \rangle$ is nef by induction, $\forall i; 1\leq i \leq k$. For $i = 1$, $\mathcal{F}_{1} \langle - \delta_{1} \rangle = \mathcal{G}_{1} \langle - \delta_{1} \rangle$ which is nef. Now, assume that $\mathcal{F}_{i-1} \langle - \delta_{i-1} \rangle$ is a nef bundle. Then $\mathcal{F}_{i-1} \langle - \delta_{i} \rangle$ is $\Q$-ample. Using the above exact sequence and since $\mathcal{F}_{i-1} \langle - \delta_{i} \rangle$ and $\mathcal{G}_{i} \langle - \delta_{i} \rangle$ are nef, we see that $\mathcal{F}_{i} \langle - \delta_{i} \rangle$ is a nef bundle. 
\par Thus, $\mathcal{O}_{\mathbb{P}_{C}(\mathcal{F}_{i} \langle - \delta_{i} \rangle)}(1)$ is a nef line bundle. So $\mathcal{O}_{\mathbb{P}_{C}(\mathcal{F}_{i} )}(1)) \otimes \pi_{i}^{*} \mathcal{O}(- \delta_{i})$ is nef, then by the remarks that $\lambda_{i}^{*}\mathcal{O}_{\mathbb{P}_{C}(\mathcal{F}_{i} )}(1) = M_{i}$ and $\lambda_{i}^{*}(\pi_{i}^{*} \mathcal{O}(- \delta_{i})) = \mathcal{O}(-\mu_{i} F)$, we conclude that 
$N_{i}$ is nef, $\forall i; 1 \leq i \leq k $. 
\end{proof}
\begin{remark}
It is noted that the above lemma can be proven by the original Xiao’s argument, leveraging the Miyaoka-Nakayama result \cite[Corollary IV.3.8]{Nakayama}. The argument presented in the lemma above serves as an alternative proof using the language of $\mathbb{Q}$-twisted vector bundles.
\end{remark}
\begin{lemma}\label{3.6}
$\forall i; 1 \leq i \leq k, \hspace{0.2cm} r_{i} = \rk \mathcal{F}_{i} \leq h^{0}(F, {N_{i}}_{|_{F}})$.
\end{lemma}
\begin{proof}
Let $\pi_{i}$ be the projection from $\mathbb{P}_{C}(\mathcal{F}_{i})$ to $C$. Then,
\[
(\pi_{i} \circ \lambda_{i})_{*}(M_{i}) = (\pi_{i})_{*}( (\lambda_{i})_{*}M_{i})
\]
\[
= (\pi_{i})_{*}(\mathcal{O}_{\mathbb{P}_{C}(\mathcal{F}_{i})}(1) \otimes (\lambda_{i})_{*} \mathcal{O}_{\Hat{S}}) \supseteq (\pi_{i})_{*}(\mathcal{O}_{\mathbb{P}_{C}(\mathcal{F}_{i})}(1)) = \mathcal{F}_{i} \text{.}
\]
Thus,
\[
r_{i} \leq h^{0}(F, {M_{i}}_{|_{F}}) \text{,}
\]
which implies
\[
r_{i} \leq h^{0}(F, {N_{i}}_{|_{F}})
.\]
\end{proof}
\begin{proposition}\label{positivity of degree}
Let $d_{i} = \deg({N_{i}}_{|_{F}}) = N_{i} . F$ such that $ 1 \leq i \leq k$. Then, \[
d_{k} \geq d_{k-1}\geq ... \geq d_{1} \geq 0
.\]
\end{proposition}
\begin{proof}
Since $F$ is a fiber, $F^{2} = 0$. Then,
\[d_{i} = N_{i} . F = (M_{i} - \mu_{i} F) . F = M_{i}.F
\]
\[
= \lambda_{i}^{*}(\mathcal{O}_{\mathbb{P}_{C}(\mathcal{F}_{i})}(1)) . F = (\epsilon^{*}(D - Z_{\mathcal{F}_{i}}) - E).F
\]
\[
= \epsilon^{*}(D - Z_{\mathcal{F}_{i}}) . F = (D - Z_{\mathcal{F}_{i}}) . F \geq 0
.\]
But $$Z_{\mathcal{F}_{i}} \geq Z_{\mathcal{F}_{i + 1 }}\text{.}$$ Thus, $$D - Z_{\mathcal{F}_{i
+ 1}} = D -Z_{\mathcal{F}_{i}} + (Z_{\mathcal{F}_{i}} - Z_{\mathcal{F}_{i + 1}}) \text{.}$$
Therefore, 
\[ d_{i + 1} \geq d_{i} \text{.}
\]
\end{proof}
\begin{proposition}\label{specialandnonspecial}
Continuing with the setting as above, we define the following constant for the data $(D, \mathcal{F})$: \[  \Hat{n}_{(D,\mathcal{F})} := \left\{
    \begin{array}{ll}
\max\{i| \hspace{0.1cm} {N_{i}}_{|_{F}} \text{is special}\} 
\\ -\infty  \text{ otherwise} \text{.}
\end{array}
\right. 
\]
If $\Hat{n}_{(D, \mathcal{F})} \neq -\infty$, then ${N_{j}}_{|_{F}}$ is special for $1 \leq j \leq \Hat{n}_{(D, \mathcal{F})}$ and ${N_{j}}_{|_{F}}$ is nonspecial for $ \Hat{n}_{(D, \mathcal{F})} +1 \leq j \leq k$. Otherwise, ${N_{j}}_{|_{F}}$ is nonspecial, $\forall j; 1\leq j \leq k$.
\end{proposition}
\begin{proof}
Recall that \[Z_{\mathcal{F}_{1}} \geq ... \geq Z_{\mathcal{F}_{k}}.\]
Also, we identified the general fiber of $S$ with the general fiber of $\Hat{S}$. Let $\Hat{n}_{(D, \mathcal{F})} + 1 \leq j \leq k -1$, thus
\[{N_{j}}_{|_{F}} = \epsilon^{*}(D - {Z_{\mathcal{F}_{j}})}_{|_{F}} = (D - {Z_{\mathcal{F}_{j}})}_{|_{F}} \leq (D - {Z_{\mathcal{F}_{j+1}})}_{|_{F}} = {N_{j+1}}_{|_{F}}.\]
Consider the following short exact sequence:
\[
0 \longrightarrow {N_{j}}_{|_F} \longrightarrow {N_{j+1}}_{|_F} \longrightarrow {N_{j+1}}_{|_F}/{N_{j}}_{|_{F}}  \longrightarrow 0
\]
which induces a long exact sequence in cohomology:
\[
\dots \longrightarrow H^{0}(F, {N_{j+1}}_{|_{F}}/{N_{j}}_{|_{F}}) \longrightarrow H^{1}(F, {N_{j}}_{|_{F}}) \longrightarrow H^{1}(F, {N_{j+1}}_{|_{F}})   \longrightarrow 0.
\]
So, if $h^{1}(F, {N_{j}}_{|_{F}} ) = 0$, then $h^{1}(F, {N_{j+1}}_{|_{F}} ) = 0$.
\end{proof}
\par To illustrate the above proposition, we consider the following examples.
\begin{example}
    Let $D = K_{S/C}, g \geq 1$, and $\mathcal{F}= f_{*}\omega_{S/C}$. Then we have ${N_{k}}_{|_{F}} \simeq K_{F} $, thus $H^{1}(F, K_{F}) = 1$. Hence, ${N_{k}}_{|_{F}}$ is special. Furthermore, it follows that ${N_{i}}_{|_{F}}$ is special $ \forall i; 1 \leq i \leq k$ and $\Hat{n}_{(K_{S/C},  f_{*}\omega_{S/C})}= k$. Conversely, for an arbitrary relatively effective divisor $D$, if $\Hat{n}_{(D, \mathcal{F})} = k$, then $\deg D_{|_{F}} \leq 2g-2$. 
\end{example}
\begin{example}\label{relativeexamplee}  Now, let $D = K_{S/C} + \Delta, g \geq 1$, $\mathcal{F}= f_{*}\mathcal{O}_{S}(K_{S/C} + \Delta)$, and $\Delta$ is an effective divisor on $S$ with $\Delta.F > 0$. Then ${N_{k}}_{|_{F}} \simeq K_{F} + {\Delta}_{|_{F}} $, thus $H^{1}(F, K_{F} + {\Delta}_{|_{F}}) = 0$ since the degree of $K_{F} + {\Delta}_{|_{F}}$ is strictly  greater than $2g - 2$. So we conclude that $\Hat{n}_{(K_{S/C} + \Delta, f_{*}\mathcal{O}_{S}(K_{S/C} + \Delta))} < k$.
\end{example}
\begin{example}
    As in  Example \ref{relativeexamplee}, let $D = K_{S/C} + \Delta$ and $\Delta$ is an effective divisor on $S$ with $\Delta.F > 0$. Assume that $f_{*}\mathcal{O}_{S}(D)$ is semi-stable. Then $k= 1$ and $\Hat{n}_{(K_{S/C} + \Delta, f_{*}\mathcal{O}_{S}(K_{S/C} + \Delta))} = -\infty$.
\end{example}
\begin{example}\label{ruledsurfaceexample}
        We give another example such that $\Hat{n}_{(D, \mathcal{F})} = -\infty$. Let $\mathcal{E} = \mathcal{O}_{\mathbb{CP}^{1}} \oplus \mathcal{O}_{\mathbb{CP}^{1}}(1)$ and $S = \mathbb{P}_{\mathbb{CP}^{1}}(\mathcal{E})$. Then we define a natural fibration $f : S \to \mathbb{CP}^{1}$. $S$ has only one negative curve $E \simeq \mathbb{CP}^{1}$ such that $E^{2} = -1$ and $E = \mathcal{O}_{S}(1) - F$ in $Pic(S)$. By construction, the trivial part $\mathcal{O}_{\mathbb{CP}^{1}}$ corresponds to a $0$-dimensional linear sub-system $L_{0}$ of $|\mathcal{O}_{S}(1)|$ generated by the effective divisor $E + F$, and the bundle $\mathcal{O}_{\mathbb{CP}^{1}}(1)$ to the $1$-dimensional linear sub-system $L_{1}$ generated by the  two effective divisor of $|\mathcal{O}_{S}(1)|$ different from $E + F$.
        Moreover, $$0 \subsetneq \mathcal{O}_{\mathbb{CP}^{1}}(1) \subsetneq \mathcal{E}$$is the Harder-Narasimhan filtration of $\mathcal{E}$. Since $L_{1}$ has no fixed part, then $N_{1} = \mathcal{O}_{S}(1) - F$ and $N_{2} = \mathcal{O}_{S}(1)$. Thus, ${N_{1}}_{|_{F}}$ and ${N_{2}}_{|_{F}}$ are nonspecial divisors and $\Hat{n}_{(\mathcal{O}_{1}(S), \mathcal{E})} = -\infty \text{.}$
\end{example}
\begin{example}
    More generally than Example \ref{ruledsurfaceexample}, every fibred surface $f : S \to C$  with $g(F) = 0$ has $\Hat{n}_{(D, \mathcal{F})} = -\infty$ for every relatively effective divisor $D$ and $\mathcal{F} \subseteq f_{*}\mathcal{O}_{S}(D)$, since there are no special divisors on $\mathbb{CP}^{1}$.
\end{example}
\par Now, it is natural to ask for some information about the sequence $(\frac{d_{i}}{h^{0}(F, {N_{i}}_{|_{F}}) -
1})_{i \in \{1,...k \}}$. For instance, is it an increasing finite sequence? Is it decreasing? Is it bounded from
below by a strictly positive number?
\begin{lemma}
$\forall i; 2 \leq i \leq k$, we have $h^{0}(F, {N_{i}}_{|_{F}}) > 1$.
\end{lemma}
\begin{proof} We have $h^{0}(F, {N_{i}}_{|_{F}}) \geq \rk(\mathcal{F}_{i})$. So if
$h^{0}(F, {N_{i}}_{|_{F}}) = 1$, then the only possibility is $i = 1$. In this case, the degree is $d_{1} = g(F) - h^{1}(F, {N_{1}}_{|_{F}}).$
\end{proof}
\par In the following theorem, we  assume that $\rk \mathcal{F} \geq 2$.
\begin{theorem}\label{3.10}
Let $f : S \to C$ be a fibred surface with general fiber $F$, and $D$ a relatively effective  divisor on $S$.
Consider the Harder-Narasimhan filtration $(\mathcal{F}_{i})$ of $\mathcal{F} \subseteq f_{*}\mathcal{O}_{S}(D)$ such that $\rk \mathcal{F} \geq 2$. Define: 
\[  \Hat{s}_{(D, \mathcal{F})}:= \left\{
    \begin{array}{ll}
1 &  \mbox{if $h^{0}(F, {N_{1}}_{|_{F}}) > 1$}  
\\ 2 & \mbox{otherwise} 
\end{array}
\right. 
\]
and 
\[
\beta_{D}:=  1 + \frac{g(F)}{h^{0}(F, {D}_{|_{F}}) - 1} \text{.}
\]
Then, the following result hold:
\begin{itemize}
\item[(1).] If $\Hat{n}_{(D, \mathcal{F})} = -\infty$, then ${N_{i}}_{|_{F}}$ is nonspecial $\forall i; 1 \leq i \leq k$, and
\[ \beta_{D}
\leq
\frac{d_{k}}{h^{0}({N_{k}}_{|_{F}}) - 1} \leq ... \leq \frac{d_{i+1}}{h^{0}({N_{i+1}}_{|
_{F}}) - 1} \leq
\]
\[
\frac{d_{i}}{h^{0}({N_{i}}_{|_{F}}) - 1} \leq ... \leq \frac{d_{\Hat{s}_{(D, \mathcal{F})}}}{h^{0}({N_{\Hat{s}_{(D, \mathcal{F})}}}_{|
_{F}}) - 1}
.\]
\item[(2).] Otherwise: 
\begin{itemize}
    \item $\forall i; \Hat{n}_{(D, \mathcal{F})} + 1 \leq i \leq k $: 
   \[ \beta_{D}
\leq
\frac{d_{k}}{h^{0}({N_{k}}_{|_{F}}) - 1} \leq ... \leq \frac{d_{i+1}}{h^{0}({N_{i+1}}_{|
_{F}}) - 1} \leq
\]
\[
\frac{d_{i}}{h^{0}({N_{i}}_{|_{F}}) - 1} \leq ... \leq \frac{d_{\Hat{n}_{(D, \mathcal{F})} + 1}}{h^{0}({N_{\Hat{n}_{(D, \mathcal{F})} + 1}}_{|
_{F}}) - 1}
.\]
\item $\forall i; \Hat{s}_{(D, \mathcal{F})} \leq i \leq \Hat{n}_{(D, \mathcal{F})}$:
 \[\frac{d_{i}}{h^{0}({N_{i}}_{|_{F}}) - 1} \geq 2 \text{.}\]
 \end{itemize}
 \end{itemize}
\end{theorem}
\begin{proof} For $(1)$, if $\Hat{n}_{(D, \mathcal{F})}=-\infty$, then by definition of $\Hat{n}_{(D, \mathcal{F})}$,  ${N_{i}}_{|_{F}}$ is nonspecial $\forall i; 1 \leq i \leq k$. Thus, by applying the Riemann-Roch formula:
\[
h^{0}({N_{i}}_{|_{F}}) = d_{i} + 1 - g(F)
\text{.}\]
Since
\[
h^{0}({N_{i+1}}_{|_{F}}) = h^{0}({N_{i}}_{|_{F}}) + d_{i+1} - d_{i}\text{,} \hspace{0.3cm} 
\]
it follows that $\forall i; \Hat{s}_{(D, \mathcal{F})} \leq i \leq k-1$:
\[ \frac{d_{i+1}}{h^{0}({N_{i+1}}_{|_{F}}) - 1} = \frac{d_{i} + d_{i+1} - d_{i}}
{h^{0}({N_{i}}_{|_{F}}) + d_{i+1} - d_{i} - 1} \leq \frac{d_{i}}{h^{0}({N_{i}}_{|_{F}}) - 1}
\text{.}\]
Moreover, $D_{|_{F}}$ is nonspecial, and then by the Riemann-Roch formula we have
\[
h^{0}(F, D_{|_{F}}) = D.F + 1 - g(F)
\text{.}\]
Thus,
\[
h^{0}(F, D_{|_{F}}) = h^{0}(F, {N_{i}}_{|_{F}}) + D.F - d_{i}\text{,}
\hspace{0.2cm} \forall i; 1 \leq i \leq k \text{.} \]
Then, we deduce the desired lower bound:
\[
\frac{d_{i}}{h^{0}({N_{i}}_{|_{F}}) - 1} \geq 1 + \frac{g(F)}{h^{0}(F, {D}_{|_{F}}) - 1} \text{,} \hspace{0.2cm} \forall i; \Hat{s}_{(D,\mathcal{F})} \leq i \leq k
.\]
This proves $(1)$.
\vspace{0.2cm}
\\For $(2)$, if $\Hat{n}_{(D, \mathcal{F})} + 1 \leq i \leq k$, then argue as in $(1)$.  Suppose now that 
 ${N_{i}}_{|_{F}}$ is special, that is $\Hat{s}_{(D,\mathcal{F})} \leq i \leq \Hat{n}$. Then by Clifford Theorem \cite{ACGH}, we have
\[
d_{i} \geq 2 (h^{0}({N_{i}}_{|_{F}}) - 1)
\]
which implies
\[
\frac{d_{i}}{h^{0}({N_{i}}_{|_{F}}) - 1} \geq 2
\text{.}\]
\end{proof}
\begin{remark}\label{degreeremark}
    If we compare  Proposition \ref{positivity of degree} and  Theorem \ref{3.10}, then we deduce that $\forall i; \Hat{s}_{(D, \mathcal{F})} \leq i \leq k$:
    \[
    d_{k} \geq d_{k-1}\geq ... \geq d_{\Hat{s}} \geq 1 \text{.}
    \]
\end{remark}
\section{Slope inequalities}
\par Now, we are ready to present the technical lemma in the method, we called it the \emph{Modified Xiao Lemma}. Note that it is a more general form of Xiao \cite[Lemma 2]{xiaogang}.

\begin{lemma}
[\emph{Modified Xiao Lemma}]\label{4.1}
Let $\Hat{f} : \Hat{S} \to C$ be a fibred surface with $F$ its general fiber, $\Hat{D}$ be a divisor on $\Hat{S}$, and suppose that there
exist a sequence of effective
divisors:
\[
\mathcal{Z}_{1} \geq \mathcal{Z}_{2} \geq ... \geq \mathcal{Z}_{j}
\text{,}\]
and a sequence of rational numbers:
\[
\mu_{1} \geq \mu_{2} \geq ... \geq \mu_{j}
\text{,}\]
such that for every $i \in \{1,...,j\}$, we have $$\mathcal{N}_{i} : = \hat{D} - \mathcal{Z}_{i} - \mu_{i}F$$
are nef $\Q$-divisors. Then,
\[
\Hat{D}^{2} \geq \sum_{i = 1}^{j-1} (d_{i} + d_{i+1})(\mu_{i} - \mu_{i+1}) + 2 \Hat{D}.\mathcal{Z}_{j} -
\mathcal{Z}_{j}^{2} + 2\mu_{j}d_{j} \text{,}
\]
where $d_{i} = \mathcal{N}_{i}.F$.
\end{lemma}
\begin{proof} First, observe that $\mathcal{N}_{1}^{2} \geq 0$ by nefness. However,
\[
\mathcal{N}_{i}^{2} = \mathcal{N}_{i}(\mathcal{N}_{i-1} + (\mathcal{Z}_{i-1} -\mathcal{Z}_{i}) + (\mu_{i-1} - \mu_{i})F)
\]
\[
\geq \mathcal{N}_{i}(\mathcal{N}_{i-1} + (\mu_{i-1} - \mu_{i})F )
\]
\[
\geq
(\mathcal{N}_{i-1} + (\mathcal{Z}_{i-1} -\mathcal{Z}_{i}) + (\mu_{i-1} - \mu_{i})F)(\mathcal{N}_{i-1} +
(\mu_{i-1} - \mu_{i})F )
\]
\[
\geq \mathcal{N}_{i-1}^{2} + (\mu_{i-1} - \mu_{i})(2\mathcal{N}_{i-1}F +(\mathcal{Z}_{i-1} -
\mathcal{Z}_{i})F )
\]
\[ =
\mathcal{N}_{i-1}^{2} + (\mu_{i-1} - \mu_{i})(d_{i-1} + d_{i}) \text{.}
\]
So, by induction we have
\[
\mathcal{N}_{j}^{2} \geq \sum_{i = 1}^{j-1} (d_{i} + d_{i+1})(\mu_{i} - \mu_{i+1}) \text{.}
\]
Hence,
\[
(\Hat{D} - \mathcal{Z}_{j} - \mu_{j}F)^{2} \geq \sum_{i = 1}^{j-1} (d_{i} + d_{i+1})(\mu_{i} - \mu_{i+1}) \text{.}
\]
But,
\[
(\Hat{D} - \mathcal{Z}_{j} - \mu_{j}F)^{2} = (\Hat{D} - \mathcal{Z}_{j})^{2} - 2\mu_{j}(\Hat{D} - \mathcal{Z}_{j})F
\]
\[
= \Hat{D}^{2} - 2\Hat{D}.\mathcal{Z}_{j} + \mathcal{Z}_{j}^{2} - 2\mu_{j}d_{j} \text{.}
\]
Thus,
\[
\Hat{D}^{2} \geq \sum_{i = 1}^{j-1} (d_{i} + d_{i+1})(\mu_{i} - \mu_{i+1}) + 2 \Hat{D}.\mathcal{Z}_{j} -
\mathcal{Z}_{j}^{2} + 2\mu_{j}d_{j} \text{.}
\]
\end{proof}
\begin{remark}
    The term $2 \Hat{D}.\mathcal{Z}_{j} - \mathcal{Z}_{j}^{2} + 2\mu_{j}d_{j}$ describes the negativity of $\Hat{D}$.
\end{remark}
\begin{example}\label{4.2}
In the setting of Lemma \ref{4.1}, suppose that $j = k $, $\Hat{D}$ is nef, and $\mu_{k} \geq 0$. Set $\mathcal{Z}_{k+1} = 0$ and $\mu_{k + 1} = 0$. Then apply  Lemma \ref{4.1} to the sequence of effective divisors: 
\[
\mathcal{Z}_{1} \geq \mathcal{Z}_{2} \geq ... \geq \mathcal{Z}_{k+1} = 0
\text{,}\]
and a sequence of rational numbers:
\[
\mu_{1} \geq \mu_{2} \geq ... \geq \mu_{k+1}=0
\text{.}\]
So, we conclude the original result of Xiao \cite[Lemma 2]{xiaogang}:
\[
\Hat{D}^{2} \geq \sum_{i = 1}^{k} (d_{i} + d_{i+1})(\mu_{i} - \mu_{i+1})
.\]
\end{example}
Prior to stating the main results, we apply  Lemma \ref{4.1} to the data $(D, \mathcal{F})$ where $\mathcal{F} \subseteq f_{*}\mathcal{O}_{S}(D)$. Let  $(\mathcal{F}_{i})_{0 \leq i \leq k}$ be the Harder-Narasimhan filtration of $\mathcal{F}$, and $(Z_{i}, M_{i}, N_{i})_{1 \leq i \leq k}$ be the triple of fixed parts, moving parts, and the Miyaoka divisors respectively (Definition \ref{miyaoka}). Then, we deduce the following inequalities.
\begin{theorem}\label{deductiontheorem}
    Let $f: S \to C$ be a fibred surface with general fiber $F$. Consider the data  $(D, \mathcal{F})$ where $D$ is a relatively effective divisor on $S$ and $\mathcal{F} \subseteq f_{*}\mathcal{O}_{S}(D)$ is a locally free sub-sheaf on $C$ with $\rk \mathcal{F}\geq 2$. Then we have the following three cases:
    \begin{itemize}
        \item[(1).] If $\Hat{n}_{(D, \mathcal{F})} = -\infty$, then
        \[D^{2} \geq 2\beta_{D} \deg \mathcal{F} - \beta_{D}\mu_{1}
+ (\beta_{D} - 2\beta_{D}r_{k} + 2d_{k})\mu_{k} + 2 \epsilon^{*}D.Z_{k} - Z_{k}^{2} \text{.}
\]
        \item[(2).] If $\Hat{n}_{(D, \mathcal{F})} = k$, then
        \[D^{2} \geq 4 \deg \mathcal{F} - 2\mu_{1}
+ 2(1 - 2r_{k} + d_{k})\mu_{k} + 2 \epsilon^{*}D.Z_{k} - Z_{k}^{2} \text{.}
\]
        \item[(3).] If $1 \leq \Hat{n}_{(D, \mathcal{F})} < k$, then
        \[
        D^{2} \geq 4\sum_{i = 1}^{\Hat{n}_{(D, \mathcal{F})}-1} r_{i}(\mu_{i} - \mu_{i+1}) + 2\beta_{D} \sum_{i = \Hat{n}_{(D, \mathcal{F})} + 1}^{k-1} r_{i}(\mu_{i} - \mu_{i+1})
        \]
        \[
        -2(\mu_{1} - \mu_{\Hat{n}_{(D, \mathcal{F})} +1}) - \beta_{D}(\mu_{\Hat{n}_{(D, \mathcal{F})} +1} -\mu_{k}) + (\beta_{D}   + 2)r_{\Hat{n}_{(D, \mathcal{F})}}(\mu_{\Hat{n}_{(D, \mathcal{F})}} - \mu_{\Hat{n}_{(D, \mathcal{F})} + 1})
        \]
        \[
       + 2 \epsilon^{*}D.{Z}_{k} -Z_{k}^{2} + 2\mu_{k}d_{k} \text{.}
        \]  
Here $\epsilon$ is constructed as in  Proposition \ref{proposition3.3} and $\beta_{D}$ is as in  Theorem \ref{3.10}.
\end{itemize}
\end{theorem}
\begin{proof}
    Recall that $\Hat{f} = f \circ \epsilon : \Hat{S} \to C$ is a fibred surface with $F$ its general fiber and $\Hat{D}:= \epsilon^{*}D$ is a relatively effective divisor on $\Hat{S}$. Consider $(\mathcal{F}_{i})_{0 \leq i \leq k}$ the Harder-Narasimhan filtration of $\mathcal{F}$, let $(\mu_{1}, \dots, \mu_{k} )$ be the sequence of slopes and $(Z_{i}, M_{i}, N_{i})_{1 \leq i \leq k}$  the triple of fixed parts, moving parts, and the Miyaoka divisors respectively, recall also that $\hat{D} = Z_{i} + M_{i}, \forall i; 1 \leq i \leq k $, and 
\[
Z_{1} \geq Z_{2} \geq \dots \geq Z_{k} \geq 0,
\]
\[
\mu_{1} > \mu_{2} > \dots > \mu_{k}\text{.}
\]
Therefore, $N_{i}$ are $\Q$-nef divisors on $\Hat{S}$, where $d_{i} = N_{i}.F$. Now, by  Lemma \ref{4.1}, we have the following inequality: 
\begin{equation}\label{eq}\Hat{D}^{2} \geq \sum_{i = 1}^{k-1} (d_{i} + d_{i+1})(\mu_{i} - \mu_{i+1}) + 2 \Hat{D}.{Z}_{k} -Z_{k}^{2} + 2\mu_{k}d_{k} \text{.}  
\end{equation}
Recall that $\epsilon : \Hat{S} \to S$ is a proper birational morphism, hence $\Hat{D}^{2} = D^{2}$. By  Theorem \ref{3.10}, we have the following natural three cases:
\begin{itemize}
    \item[(1).] If $\Hat{n}_{(D, \mathcal{F})} = -\infty$, then ${N_{i}}_{|_{F}}$ is nonspecial $\forall i; 1 \leq i \leq k$, and
    \[
    d_{i} \geq \beta_{D}(h^{0}( {N_{i}}_{|_{F}} )-1) \geq \beta_{D}(r_{i}-1)\text{.}
    \]
    The right hand side inequality follows from  Lemma \ref{3.6}, where $r_{i} = \rk(\mathcal{F}_{i}), \forall i; 1 \leq i \leq k$, we also recall that $r_{i+1} \geq r_{i} - 1, \forall i; 1 \leq i \leq k-1$, hence
    \[
    d_{i+1} \geq \beta_{D}r_{i} \text{.}
    \]
     We substitule all these information into \eqref{eq}: 
\[
D^{2} \geq 2 \beta_{D} \sum_{i = 1}^{k-1} r_{i}(\mu_{i} - \mu_{i+1})
- \beta_{D}(\mu_{1} - \mu_{k})  + 2\epsilon^{*}D.Z_{k} - Z_{k}^{2} + 2\mu_{k}d_{k}\text{.}
\]
By Lemma \ref{degreelemma}, we deduce the desired inequality:
\[D^{2} \geq 2\beta_{D} \deg \mathcal{F} - \beta_{D}\mu_{1}
+ (\beta_{D} - 2\beta_{D}r_{k} + 2d_{k})\mu_{k} + 2 \epsilon^{*}D.Z_{k} - Z_{k}^{2} \text{.}
\]
 \item[(2).] If $\Hat{n}_{(D, \mathcal{F})} = k$, then ${N_{i}}_{|_{F}}$ is special $\forall i; 1 \leq i \leq k$ and
  \[
    d_{i} \geq 2(h^{0}( {N_{i}}_{|_{F}} )-1) \geq 2(r_{i}-1)
    \text{,}\]
    however
    \[
    d_{i+1} \geq 2r_{i} \text{.}
    \]
This implies
\[D^{2} \geq 4 \deg \mathcal{F} - 2\mu_{1}
+ 2(1 - 2r_{k} + d_{k})\mu_{k} + 2 \epsilon^{*}D.Z_{k} - Z_{k}^{2} \text{.}
\]
\item[(3).] If $1 \leq \Hat{n}_{(D, \mathcal{F})} < k$, then 
\[
\forall i; \Hat{n}_{(D, \mathcal{F})} + 1 \leq i \leq k,\hspace{0.4cm} d_{i} \geq \beta_{D}(h^{0}( {N_{i}}_{|_{F}} )-1) \geq \beta_{D}(r_{i}-1) \text{,}
\]
and 
\[
\forall i;  1 \leq i \leq \Hat{n}_{(D, \mathcal{F})}, \hspace{0.4cm} d_{i} \geq 2(h^{0}( {N_{i}}_{|_{F}} )-1) \geq 2(r_{i}-1) \text{.}
\]
So, we decompose the right hand side of \eqref{eq} into three parts:
\[
D^{2} \geq \sum_{i = 1}^{\Hat{n}_{(D, \mathcal{F})}-1} (d_{i} + d_{i+1})(\mu_{i} - \mu_{i+1}) +\sum_{i = \Hat{n}_{(D, \mathcal{F})}+1}^{k-1} (d_{i} + d_{i+1})(\mu_{i} - \mu_{i+1}) \]
\[
    + (d_{\Hat{n}_{(D, \mathcal{F})}} + d_{\Hat{n}_{(D, \mathcal{F})} +1})(\mu_{\Hat{n}_{(D,\mathcal{F})}} - \mu_{\Hat{n}_{(D,\mathcal{F})}+1})      
\]
\[+ 2 \Hat{D}.{Z}_{k} -Z_{k}^{2} + 2\mu_{k}d_{k} \text{.}
\]
Which implies
\[
D^{2} \geq 4\sum_{i = 1}^{\Hat{n}_{(D, \mathcal{F})}-1} r_{i}(\mu_{i} - \mu_{i+1}) + 2\beta_{D} \sum_{i = \Hat{n}_{(D, \mathcal{F})} + 1}^{k-1} r_{i}(\mu_{i} - \mu_{i+1})
        \]
        \[
        -2(\mu_{1} - \mu_{\Hat{n}_{(D, \mathcal{F})}}) - \beta_{D}(\mu_{\Hat{n}_{(D, \mathcal{F})} +1} -\mu_{k}) + (\beta_{D}   + 2)r_{\Hat{n}_{(D, \mathcal{F})}}(\mu_{\Hat{n}_{(D, \mathcal{F})}} - \mu_{\Hat{n}_{(D, \mathcal{F})} + 1})
        \]
        \[
       + 2 \Hat{D}.{Z}_{k} -Z_{k}^{2} + 2\mu_{k}d_{k} 
      \text{,}  \]
since 
$$d_{\Hat{n}_{(D, \mathcal{F})}} \geq 2(r_{\Hat{n}_{(D, \mathcal{F})}} -1)\text{,} \hspace{0.2cm} \text{and} \hspace{0.2cm} d_{\Hat{n}_{(D, \mathcal{F})} + 1} \geq \beta_{D}(r_{\Hat{n}_{(D, \mathcal{F})} +1} -1) \geq \beta_{D}(r_{\Hat{n}_{(D, \mathcal{F})}}) \text{.}$$ 
\end{itemize}
\end{proof}
\begin{remark}
The lower bound  that we obtain for $D^{2}$ in  Theorem \ref{deductiontheorem} for $1 \leq \Hat{n}_{(D, \mathcal{F})} < k$ is given by a sum of three parts. The first part,  
$$4\sum_{i = 1}^{\Hat{n}_{(D, \mathcal{F})}-1} r_{i}(\mu_{i} - \mu_{i+1})$$ describes the effect of special Miyaoka divisors. The second part,
$$2\beta_{D} \sum_{i = \Hat{n}_{(D, \mathcal{F})} + 1}^{k-1} r_{i}(\mu_{i} - \mu_{i+1})$$ explains the impact of the nonspecial Miyaoka divisors in the sequence $(N_{i})_{1\leq i \leq k}$. The third part is a  transition from special to nonspecial Miyaoka divisors. In general, we could give a lower bound for $D^{2}$ by $$4\deg \mathcal{F}_{\Hat{n}_{(D,\mathcal{F})}}, \hspace{0.2cm} 2\beta_{D}\deg \mathcal{F}/\mathcal{F}_{\Hat{n}_{(D,\mathcal{F})}}$$ and other terms, but to formulate the slope inequality for the data $(D, \mathcal{F})$, it turns out that we should first start by taking  $\min(2, \beta_{D})$ and write, for instance, $$d_{i} \geq \min(2, \beta_{D})(r_{i} -1)\text{.}$$ Following this set-up, we will obtain in the next corollary a lower bound for $D^{2}$ by $$2 \min(2, \beta_{D})\deg \mathcal{F}$$ and other terms that we will explore in the upcoming paragraphs.
\end{remark}
\begin{defn}\label{alphadefinition}
Let $f: S \to C$ be a fibred surface, $F$ its general fiber. Consider the data $(D, \mathcal{F})$ as before, where $D$ is a relatively effective divisor on $S$ and $\mathcal{F} \subseteq f_{*}\mathcal{O}_{S}(D)$ with $\rk \mathcal{F}\geq 2$. Then we define the following number:
\[ \alpha_{(D,\mathcal{F})} := \left\{
    \begin{array}{ll}
\beta_{D}  & \mbox{if} \hspace{0.2cm}  \Hat{n}_{(D,\mathcal{F})} = -\infty 
\\ 2 & \mbox{if} \hspace{0.2cm}  \Hat{n}_{(D,\mathcal{F})} = k 
\\ \min(2, \beta_{D})& 
\mbox{if} \hspace{0.2cm} 1 \leq \Hat{n}_{(D, \mathcal{F})} < k.
\end{array}
\right. 
\]
\end{defn}

\begin{corollary}\label{Cor slope}

Let $f: S \to C$ be a fibred surface with general fiber $F$. Consider the data  $(D, \mathcal{F})$ where $D$ is a relatively effective divisor on $S$ and $\mathcal{F} \subseteq f_{*}\mathcal{O}_{S}(D)$ is a locally free sub-sheaf on $C$. 
\begin{itemize}
    \item[(1).]  If $\rk \mathcal{F} \geq 2$, then
\[D^{2} \geq 2\alpha_{(D,\mathcal{F})} \deg \mathcal{F} - \alpha_{(D,\mathcal{F})}\mu_{1}
+ (\alpha_{(D,\mathcal{F})} - 2\alpha_{(D,\mathcal{F})} r_{k} + 2d_{k})\mu_{k} + 2 \epsilon^{*}D.Z_{k} - Z_{k}^{2} \text{.}
\]
Again, here $\epsilon$ is constructed as in Proposition \ref{proposition3.3}. 
\item[(2).] Else, if $\rk \mathcal{F} =1$,  then 
\[
D^{2} \geq 2d_{1} \deg \mathcal{F} + 2 \epsilon^{*}D.Z_{1} - Z_{1}^{2} \text{.}
\]
\end{itemize}

\end{corollary}
 \begin{proof}
    First, assuming that $\rk \mathcal{F} \geq 2$, 
    for the cases $\Hat{n}_{(D, \mathcal{F})} = -\infty, k$, the inequality is proved in  Theorem \ref{deductiontheorem}. \par We only need to  prove the case $1 \leq \Hat{n}_{(D, \mathcal{F})} < k$. By Theorem \ref{3.10} and Lemma \ref{3.6}, we have  
\[ d_{i} \geq \alpha_{(D,\mathcal{F})} (h^{0}({N_{i}}_{|_{F}} )-1) \hspace{0.2cm} \forall i;  1 \leq i \leq k  \text{.}\]
Putting this last inequality into inequality  \eqref{eq},  we obtain: 
\[
D^{2} \geq 2\alpha_{(D,\mathcal{F})} \sum_{i = 1}^{k-1} r_{i}(\mu_{i} - \mu_{i+1})
- \alpha_{(D,\mathcal{F})}(\mu_{1} - \mu_{k})  + 2\epsilon^{*}D.Z_{k} - Z_{k}^{2} + 2\mu_{k}d_{k}\text{.}
\]
By Lemma \ref{degreelemma}, we deduce the desired inequality:
\[D^{2} \geq 2\alpha_{(D,\mathcal{F})} \deg \mathcal{F} - \alpha_{(D,\mathcal{F})}\mu_{1}
+ (\alpha_{(D,\mathcal{F})} - 2\alpha_{(D,\mathcal{F})} r_{k} + 2d_{k})\mu_{k} + 2 \epsilon^{*}D.Z_{k} - Z_{k}^{2} \text{.}
\]
Now, if $\rk \mathcal{F} = 1$, then $\mathcal{F}$ is a locally free sheaf of rank $1$, $k = 1$, and $\mu_{k} = \mu_{1}$.  Therefore, $$h^{0}({N_{1}}_{|_{F}}) = r_{1} = 1, \hspace{0.2cm} d_{k} = d_{1} = g(F) - h^{1}({N_{1}}_{|_{F}})\text{.}$$ By the inequality \eqref{eq}, which is a consequence of Lemma \ref{4.1} and does not require the assumption that $\mathcal{F}$ has rank at least $2$, we have 
\[
D^{2} \geq 2d_{1}\deg \mathcal{F} + 2 \epsilon^{*}D.{Z}_{1} - Z_{1}^{2} \text{.}
\]
 \end{proof}
 \begin{remark}\label{remark4.6}
     If $\mathcal{F}$ is a semi-stable locally free sheaf, then by Corollary \ref{Cor slope}, we have the following inequality:
      \[
D^{2} \geq 2\frac{d_{1}}{r_{1}}\deg \mathcal{F} + 2 \epsilon^{*}D.{Z}_{1} - Z_{1}^{2} \text{.}
\]
     \end{remark}
 \par Now, the goal is to discuss the first point in  Corollary \ref{Cor slope} with respect to $\mu_{1}$ and $\mu_{k}$. Without loss of generality, we assume that $\mathcal{F}$ is not a semi-stable locally free sheaf in the next theorem.
 
 \begin{theorem}\label{main result}
     Let $f: S \to C$ be a fibred surface. Consider the data $(D, \mathcal{F})$, where $D$ is a relatively effective divisor and $\mathcal{F} \subseteq f_{*}\mathcal{O}_{S}(D)$ is a locally free sub-sheaf on $C$. Assume that $\mathcal{F}$ is not semi-stable. Here, we set:
\[ t_{(D, \mathcal{F})} := \left\{
    \begin{array}{ll}
\max\{i| \mu_{i} \geq 0\}  & \mbox{if $\mu_{1} \geq 0$}  
\\ -\infty & \mbox{otherwise} \text{.}
\end{array}
\right. 
\]  
\begin{itemize}
  \item[(1).] If $t_{(D,\mathcal{F})} = 1$, then
  \[
  D^{2} \geq \frac{2d_{1}}{r_{1}} \deg \mathcal{F}_{1} + 2 \epsilon^{*}D.Z_{1} - Z_{1}^{2} \geq \frac{2d_{1}}{r_{1}} \deg \mathcal{F} + 2 \epsilon^{*}D.Z_{1} - Z_{1}^{2} \text{.} 
  \]  
    \item[(2).] If $1 < t_{(D,\mathcal{F})} \leq k$, then
  \[
  D^{2} \geq \frac{2\alpha_{(D,\mathcal{F}_{t_{(D,\mathcal{F})}})} d_{t_{(D, \mathcal{F})}}}{d_{t_{(D,\mathcal{F})}} + \alpha_{(D,\mathcal{F}_{t_{(D, \mathcal{F})}})}} \deg \mathcal{F}_{t_{(D,\mathcal{F})}} + 2 \epsilon^{*}D.Z_{t_{(D,\mathcal{F})}} - Z_{t_{(D,\mathcal{F})}}^{2}
  \]
  \[\geq \frac{2\alpha_{(D,\mathcal{F}_{t_{(D,\mathcal{F})}})} d_{t_{(D,\mathcal{F})}}}{d_{t_{(D,\mathcal{F})}} + \alpha_{(D,\mathcal{F}_{t_{(D,\mathcal{F})}})}} \deg \mathcal{F} + 2 \epsilon^{*}D.Z_{t_{(D,\mathcal{F})}} - Z_{t_{(D,\mathcal{F})}}^{2}\text{.} 
  \]  
    \item[(3).] If $t_{(D,\mathcal{F})} = -\infty$, set:
  \[ C_{(D,\mathcal{F})} := \left\{
    \begin{array}{ll}
\frac{2\alpha_{(D,\mathcal{F})} d_{k}}{-\alpha_{(D,\mathcal{F})} + 2\alpha_{(D,\mathcal{F})} r_{k} - d_{k}}  & \mbox{if $\alpha_{(D,\mathcal{F})} - 2\alpha_{(D,\mathcal{F})} r_{k} + 2d_{k} \leq 0$} 
\\ 3\alpha_{(D,\mathcal{F})} + 2d_{k} - 2\alpha_{(D,\mathcal{F})} r_{k} & \mbox{otherwise} \text{.}
\end{array}
\right. 
\]  
Then
    \[
D^{2} \geq C_{(D,\mathcal{F})}. \deg \mathcal{F} + 2 \epsilon^{*}D.Z_{k} - Z_{k}^{2} \text{.} 
    \]
\end{itemize}
In addition, if $d_{k} = \alpha_{(D,\mathcal{F})} (r_{k} - 1)$, then we have the following inequality which is independent of $t_{(D,\mathcal{F})}$:
\[
D^{2} \geq \frac{2d_{k}}{r_{k}} \deg \mathcal{F} + 2 \epsilon^{*}D.Z_{k} - Z_{k}^{2} \text{.} 
\]
 \end{theorem}
\begin{proof}
We recall that $(\mathcal{F}_{i})_{1 \leq i \leq k}$ is the sequence of  Harder-Narasimhan filtration of $\mathcal{F}$. If we take an integer $m$ such that $1 \leq m \leq k$, then  $(\mathcal{F}_{i})_{1 \leq i \leq m}$ is the sequence of  Harder-Narasimhan filtration of $\mathcal{F}_{m}$. By  Corollary \ref{Cor slope}, we have the following inequality for the data $(D, \mathcal{F}_{m})$ such that $\rk(\mathcal{F}_{m}) \geq 2$:
\begin{equation}\label{eq1}
    D^{2} \geq 2\alpha_{(D,\mathcal{F}_{m})} \deg \mathcal{F}_{m} - \alpha_{(D,\mathcal{F}_{m})}\mu_{1}
+ (\alpha_{(D,\mathcal{F}_{m})} - 2\alpha_{(D,\mathcal{F}_{m})} r_{m} + 2d_{m})\mu_{m} + 2 \epsilon^{*}D.Z_{m} - Z_{m}^{2} \text{.} 
\end{equation}
\begin{itemize}
    \item[(1).] If $t_{(D,\mathcal{F})} = 1$, we have $\mu_{1} \geq 0$ and starting from $i =2$, $\mu_{i} < 0$. Thus, we apply Remark \ref{remark4.6} for $\mathcal{F}_{1}$:
  \[
  D^{2} \geq \frac{2d_{1}}{r_{1}} \deg \mathcal{F}_{1} + 2 \epsilon^{*}D.Z_{1} - Z_{1}^{2} \geq \frac{2d_{1}}{r_{1}} \deg \mathcal{F} + 2 \epsilon^{*}D.Z_{1} - Z_{1}^{2} \text{.} 
  \]
\item[(2).] If $t_{(D,\mathcal{F})} > 1$, we have  $\mu_{1}, \dots, \mu_{t_{(D,\mathcal{F})}} \geq 0$. Therefore, by considering that
\[
\alpha_{(D,\mathcal{F}_{t_{(D,\mathcal{F})}})} - 2\alpha_{(D,\mathcal{F}_{t_{(D,\mathcal{F})}})} r_{t_{(D,\mathcal{F})}} + 2d_{t_{(D,\mathcal{F})}} \geq -\alpha_{(D,\mathcal{F}_{t_{(D,\mathcal{F})}})} \text{,}
\]
and using inequality \eqref{eq1},  we deduce  the following inequality  for $m =t_{(D,\mathcal{F})}$:
\begin{equation}\label{eq2}
    D^{2} \geq 2\alpha_{(D,\mathcal{F}_{t_{(D,\mathcal{F})}})} \deg \mathcal{F}_{t_{(D,\mathcal{F})}} - \alpha_{(D,\mathcal{F}_{t_{(D,\mathcal{F})}})} (\mu_{1} + \mu_{t_{(D,\mathcal{F})}})
+ 2 \epsilon^{*}D.Z_{t_{(D,\mathcal{F})}} - Z_{t_{(D,\mathcal{F})}}^{2} \text{.} 
\end{equation}
Now, we apply  Lemma \ref{4.1} to the sequences $\{Z_{1}, Z_{t_{(D,\mathcal{F})}}\}$ and $\{\mu_{1}, \mu_{t_{(D,\mathcal{F})}}\}$. Thus,
\begin{equation}\label{eq3}
    D^{2} \geq d_{t_{(D,\mathcal{F})}}(\mu_{1} + \mu_{t_{(D,\mathcal{F})}}) + d_{1}(\mu_{1} - \mu_{t_{(D,\mathcal{F})}}) + 2 \epsilon^{*}D.Z_{t_{(D,\mathcal{F})}} - Z_{t_{(D,\mathcal{F})}}^{2} \text{.} 
\end{equation}
Since $d_{1}(\mu_{1} - \mu_{t_{(D,\mathcal{F})}}) \geq 0$, and  by Remark \ref{degreeremark},  we have $d_{t_{(D,\mathcal{F})}} \geq 1$. Then we can divide by $d_{t_{(D,\mathcal{F})}}$, thus the following inequality follows:
\[
- \alpha_{(D,\mathcal{F}_{t_{(D,\mathcal{F})}})} (\mu_{1} + \mu_{t_{(D,\mathcal{F})}}) \geq \frac{-\alpha_{(D,\mathcal{F}_{t_{(D,\mathcal{F})}})}}{d_{t_{(D,\mathcal{F})}}} D^{2} + 2\frac{\alpha_{(D,\mathcal{F}_{t_{(D,\mathcal{F})}})}}{d_{t_{(D,\mathcal{F})}}} \epsilon^{*}D.Z_{t_{(D,\mathcal{F})}} - \frac{\alpha_{(D,\mathcal{F}_{t_{(D,\mathcal{F})}})}}{d_{t_{(D,\mathcal{F})}}}Z_{t_{(D,\mathcal{F})}}^{2} \text{.}
\]
Combining this last inequality and the inequality \eqref{eq2}, we deduce the desired result: 
 \[
  D^{2} \geq \frac{2\alpha_{(D,\mathcal{F}_{t_{(D,\mathcal{F})}})} d_{t_{(D, \mathcal{F})}}}{d_{t_{(D,\mathcal{F})}} + \alpha_{(D,\mathcal{F}_{t_{(D, \mathcal{F})}})}} \deg \mathcal{F}_{t_{(D,\mathcal{F})}} + 2 \epsilon^{*}D.Z_{t_{(D,\mathcal{F})}} - Z_{t_{(D,\mathcal{F})}}^{2}
  \]
  \[\geq \frac{2\alpha_{(D,\mathcal{F}_{t_{(D,\mathcal{F})}})} d_{t_{(D,\mathcal{F})}}}{d_{t_{(D,\mathcal{F})}} + \alpha_{(D,\mathcal{F}_{t_{(D,\mathcal{F})}})}} \deg \mathcal{F} + 2 \epsilon^{*}D.Z_{t_{(D,\mathcal{F})}} - Z_{t_{(D,\mathcal{F})}}^{2}\text{.} 
  \]  
\item[(3).] If $t_{(D,\mathcal{F})} = - \infty$, so both $\mu_{1}, \mu_{k} < 0$. By the fact that 
\[
\alpha_{(D,\mathcal{F})} - 2\alpha_{(D,\mathcal{F})} r_{k} + 2d_{k} \geq -\alpha_{(D,\mathcal{F})} \text{,}
\]
we conclude that
\[
(\alpha_{(D,\mathcal{F})} - 2\alpha_{(D,\mathcal{F})} r_{k} + 2d_{k})\mu_{1} \leq -\alpha_{(D,\mathcal{F})} \mu_{1} \text{.}
\]
We replace this last inequality in \eqref{eq1} for the data $(D, \mathcal{F})$, thus
\[D^{2} \geq 2\alpha_{(D,\mathcal{F})} \deg \mathcal{F} + (\alpha_{(D,\mathcal{F})} - 2\alpha_{(D,\mathcal{F})} r_{k} + 2d_{k})(\mu_{1} + \mu_{k}) + 2 \epsilon^{*}D.Z_{k} - Z_{k}^{2} \text{.}
\]
We set 
$$A := \alpha_{(D,\mathcal{F})} - 2\alpha_{(D,\mathcal{F})} r_{k} + 2d_{k}\text{.}$$ To obtain the desired result,  we discuss the above  inequality depending on whether $A \leq 0$ or $A >0$.  
\begin{itemize}
    \item[(a).] If $A \leq 0$, then we apply  Lemma \ref{4.1} to the sequences $\{Z_{1}, Z_{k}\}$ and $\{\mu_{1}, \mu_{k}\}$:
\begin{equation}\label{eq4}
    D^{2}  \geq d_{k}(\mu_{1} + \mu_{k}) + 2\epsilon^{*}D.Z_{k} - Z_{k}^{2} \text{.}
\end{equation}
Combining the two last inequalities above, we deduce
\[D^{2} \geq \frac{2\alpha_{(D,\mathcal{F})} d_{k}}{-\alpha_{(D,\mathcal{F})} + 2\alpha_{(D,\mathcal{F})} r_{k} - d_{k}}\deg \mathcal{F} + 2\epsilon^{*}D.Z_{k} - Z_{k}^{2} \text{.}\]
\item[(b).] If $A > 0 $, we note that
\[
 \deg \mathcal{F} \leq \mu_{1} + \mu_{k} \text{.}
\]
Then,
\[
D^{2} \geq (3\alpha_{(D,\mathcal{F})} + 2d_{k} - 2\alpha_{(D,\mathcal{F})} r_{k}) \deg \mathcal{F} + 2 \epsilon^{*}D.Z_{k} - Z_{k}^{2} \text{.}
\]
\end{itemize}
\end{itemize}
 Now, we prove the last point of the theorem. If $d_{k} = \alpha_{(D,\mathcal{F})} (r_{k} - 1)$, we apply  inequality \eqref{eq1} for $m =k$:
\[D^{2} \geq 2\alpha_{(D,\mathcal{F})} \deg \mathcal{F} - \alpha_{(D,\mathcal{F})} (\mu_{1} + \mu_{k})
+ 2 \epsilon^{*}D.Z_{k} - Z_{k}^{2} \text{.}
\]
Combining this last inequality and  inequality \eqref{eq4}, we deduce the desired result: 
\[
D^{2} \geq \frac{2d_{k}}{r_{k}} \deg \mathcal{F} + 2 \epsilon^{*}D.Z_{k} - Z_{k}^{2} \text{.} 
\]
\end{proof}
\begin{proposition}\label{propositionhit}
    Let $f : S \to C$ be a fibred surface, and let $D$ be a relatively effective and relatively nef divisor  on $S$. Consider the data $(D, f_{*}\mathcal{O}_{S}(D))$. Then the following number is nonnegative:
    \[\epsilon^{*}D.Z_{k} \geq 0 \text{.}\]
\end{proposition}
\begin{proof}
 We recall that $Z_{k}= \epsilon^{*}Z_{\mathcal{F}_{k}} + E$, where $Z_{\mathcal{F}_{k}}$ is the fixed part of $|D + f^{*}A|$ and $A$ is a sufficiently very ample divisor on $C$. Thus, $Z_{\mathcal{F}_{k}}$ is an effective divisor supported on fibers. By assumption, $D$ is relatively nef, then we conclude the desired nonnegativity:
    \[
    \epsilon^{*}D.Z_{k}  =  \epsilon^{*}D.(\epsilon^{*}Z_{\mathcal{F}_{k}} + E) = \epsilon^{*}D. \epsilon^{*}Z_{\mathcal{F}_{k}} = D.Z_{\mathcal{F}_{k}} \geq 0
\text{.}    \]
    \end{proof}
    \begin{proposition}\label{proposition fixed part}
        Let $f : S \to C$ be a fibred surface and $D$ be a relatively effective divisor. Consider the data $(D, f_{*}\mathcal{O}_{S}(D))$. Then $Z_{k}^{2} \leq 0$.
    \end{proposition}
     \begin{proof}
 Our aim is to prove that $Z_{\mathcal{F}_{k}}^{2} \leq 0$. By contradiction, we  assume that $Z_{\mathcal{F}_{k}}^{2} > 0 $. Then by the Hodge index theorem,  either  $$\lim_{n \to \infty} h^{0}(nZ_{\mathcal{F}_{k}}) = +\infty \hspace{0.2cm}
\text{or} \hspace{0.2cm} \lim_{n \to \infty} h^{0}(K_{S} - nZ_{\mathcal{F}_{k}}) = +\infty \text{.}$$ However, the divisor $K_{S} - nZ_{\mathcal{F}_{k}}$ is never effective if $Z_{\mathcal{F}_{k}}$ is not zero and $n \gg 0$. Thus, the only possibility is  $$\lim_{n \to \infty} h^{0}(nZ_{\mathcal{F}_{k}}) = +\infty \text{.}$$ 
This contradicts  the fact that  $Z_{\mathcal{F}_{k}}$ is not movable. Therefore, we deduce $Z_{\mathcal{F}_{k}}^{2} \leq 0$. Consequently, by the definition of $Z_{k}$: 
\[
Z_{k}= \epsilon^{*}Z_{\mathcal{F}_{k}} + E \text{,}
\]
we conclude the nonpositivity of $Z_{k}^{2}$; in other words, $Z_{k}^{2} \leq 0$.
\end{proof}
\begin{corollary}\label{secondmainresult}
    Let $f : S \to C$ be a fibred surface and $D$ be a relatively effective and relatively nef divisor on $S$. Consider the data $(D, f_{*}\mathcal{O}_{S}(D))$.
    \begin{itemize}
        \item[(1).] If $f_{*}\mathcal{O}_{S}(D)$ is semi-stable, then: 
          \[
D^{2} \geq 2\frac{D.F}{h^{0}(F, {D}_{|_{F}})}\deg f_{*}\mathcal{O}_{S}(D)  \text{.}
\]
 \item[(2).] If $f_{*}\mathcal{O}_{S}(D)$ is not semi-stable, we have the following cases:
 \begin{itemize}
  \item[(a).] If $t_{(D,f_{*}\mathcal{O}_{S}(D))} = k$: $D$ is nef and
     \[
     D^{2} \geq \frac{2\alpha_{(D, f_{*}\mathcal{O}_{S}(D))} D.F}{D.F + \alpha_{(D, f_{*}\mathcal{O}_{S}(D))}} \deg f_{*}\mathcal{O}_{S}(D)
\text{.}     \]
     \item[(b).] If $t_{(D,f_{*}\mathcal{O}_{S}(D))} = 1$: 
      \[
     D^{2} \geq \frac{2d_{1}}{r_{1}}\deg f_{*}\mathcal{O}_{S}(D) + 2 \epsilon^{*}D.Z_{1} - Z_{1}^{2} \text{.}
     \]
     \item[(c).] If $1 < t_{(D,f_{*}\mathcal{O}_{S}(D))} < k$: \[
     D^{2} \geq \frac{2\alpha_{(D,\mathcal{F}_{t_{(D,f_{*}\mathcal{O}_{S}(D))}})} d_{t_{(D,f_{*}\mathcal{O}_{S}(D))}}}{d_{t_{(D, f_{*}\mathcal{O}_{S}(D))}} + \alpha_{(D,\mathcal{F}_{t_{(D,f_{*}\mathcal{O}_{S}(D))}})}} \deg f_{*}\mathcal{O}_{S}(D) + 2\epsilon^{*}D.Z_{t_{(D,f_{*}\mathcal{O}_{S}(D))}} - Z_{t_{(D, f_{*}\mathcal{O}_{S}(D))}}^{2}\text{.} 
     \]
     \item[(d).] If $t_{(D, f_{*}\mathcal{O}_{S}(D))} = -\infty$:
     \[D^{2} \geq C_{(D, f_{*}\mathcal{O}_{S}(D))}. \deg f_{*}\mathcal{O}_{S}(D) \text{.}\]
 \end{itemize}
        
    \end{itemize}
    In addition, if $ D.F = \alpha_{(D, f_{*}\mathcal{O}_{S}(D))}(h^{0}(F, D_{|_{F}}) - 1)$, then independently of $t_{(D, f_{*}\mathcal{O}_{S}(D))}$ we have:
     \[
D^{2} \geq 2\frac{D.F}{h^{0}(F, {D}_{|_{F}})}\deg f_{*}\mathcal{O}_{S}(D)  \text{.}
\]
\end{corollary}
\begin{proof}
    First, if $f_{*}\mathcal{O}_{S}(D)$ is semi-stable, then by Remark \ref{remark4.6} we have: 
     \[
D^{2} \geq 2\frac{d_{1}}{r_{1}}\deg f_{*}\mathcal{O}_{S}(D) + 2 \epsilon^{*}D.{Z}_{1} - Z_{1}^{2} \text{.}
\]
Since $2 \epsilon^{*}D.{Z}_{1} - Z_{1}^{2} \geq 0$ by Proposition \ref{propositionhit} and Proposition \ref{proposition fixed part}, where $d_{1} = D.F$ and $r_{1} = \rk (f_{*}\mathcal{O}_{S}(D)) = h^{0}(F, {D}_{|_{F}})$, we deduce the desired inequality:
 \[
D^{2} \geq 2\frac{D.F}{h^{0}(F, {D}_{|_{F}})}\deg f_{*}\mathcal{O}_{S}(D)  \text{.}
\]
Now, if $f_{*}\mathcal{O}_{S}(D)$ is not  semi-stable, then we apply  Theorem \ref{main result} and we have the following cases:
\begin{itemize}
    \item[(a).] If $t_{(D,f_{*}\mathcal{O}_{S}(D))} = k$:
    \[
    D^{2} \geq \frac{2\alpha_{(D, f_{*}\mathcal{O}_{S}(D))} d_{k}}{d_{k} + \alpha_{(D, f_{*}\mathcal{O}_{S}(D))}} \deg f_{*}\mathcal{O}_{S}(D) + 2 \epsilon^{*}D.Z_{k} - Z_{k}^{2} \text{.} 
    \]
    \item[(b).] If $t_{(D,f_{*}\mathcal{O}_{S}(S))} = 1$:
  \[
  D^{2}  \geq \frac{2d_{1}}{r_{1}} \deg f_{*}\mathcal{O}_{S}(D)  + 2 \epsilon^{*}D.Z_{1} - Z_{1}^{2} \text{.} 
  \]  
    \item[(c).] If $1 < t_{(D, f_{*}\mathcal{O}_{S}(D))} < k$:
  \[ D^{2} \geq \frac{2\alpha_{(D,\mathcal{F}_{t_{(D,f_{*}\mathcal{O}_{S}(D))}})} d_{t_{(D,f_{*}\mathcal{O}_{S}(D))}}}{d_{t_{(D, f_{*}\mathcal{O}_{S}(D))}} + \alpha_{(D,\mathcal{F}_{t_{(D,f_{*}\mathcal{O}_{S}(D))}})}} \deg f_{*}\mathcal{O}_{S}(D) + 2\epsilon^{*}D.Z_{t_{(D,f_{*}\mathcal{O}_{S}(D))}} - Z_{t_{(D, f_{*}\mathcal{O}_{S}(D))}}^{2}\text{.} 
     \]
    \item[(d).] If $t_{(D,f_{*}\mathcal{O}_{S}(D))} = -\infty$:
       \[
    D^{2} \geq C_{(D,f_{*}\mathcal{O}_{S}(D))}.\deg f_{*}\mathcal{O}_{S}(D) + 2 \epsilon^{*}D.Z_{k} - Z_{k}^{2} \text{.} 
    \]
\end{itemize}
By Proposition \ref{propositionhit} and Proposition \ref{proposition fixed part}, we have that $2 \epsilon^{*}D.Z_{k} - Z_{k}^{2} \geq 0$. Furthermore,  $d_{k} = D.F$. So we deduce the desired result:
 \begin{itemize}
     \item If $t_{(D, f_{*}\mathcal{O}_{S}(D))} = k$, equivalently $\mu_{k} \geq 0$, then we have 
     \[
     D^{2} \geq \frac{2\alpha_{(D, f_{*}\mathcal{O}_{S}(D))} D.F}{D.F + \alpha_{(D, f_{*}\mathcal{O}_{S}(D))}} \deg f_{*}\mathcal{O}_{S}(D) \text{.}
     \]
     \item If $t_{(D,f_{*}\mathcal{O}_{S}(D))} = -\infty$, it is equivalent to  $\mu_{1} < 0$. Then
     \[D^{2} \geq C_{(D,f_{*}\mathcal{O}_{S}(D))}. \deg f_{*}\mathcal{O}_{S}(D) \text{.}\]
 \end{itemize}
 Also, if $\mu_{k} \geq 0$, then $f_{*}\mathcal{O}_{S}(D)$ is nef on $C$, so $M_{k}$ is nef on $\Hat{S}$. Since
 \[
 \Hat{D} = Z_{k} + M_{k} \text{,} \]
$Z_{\mathcal{F}_{k}}$ is supported in the fibers, and by assumption $D$ is relatively nef, thus we deduce that $D$ is nef.
\end{proof}
\begin{remark}
    In general, if $D$ is nef, we do not necessarily have $\mu_{k} \geq 0$. 
\end{remark}
\begin{example}\label{remarkslopepaper}
    In \cite[Thoerem 5]{remarkonslope}, the authors proved that if $D$ is a relatively nef divisor on $S$ such that $D_{|{_F}}$ is generated by global sections on a general  fiber $F$ of $f$,  $D_{|{_F}}$ is a special divisor on $F$, and 
    \begin{equation}\label{Condition}
    2h^{0}(F, D_{|_{F}}) - D.F - 1 > 0 \tag{P} \text{,}
    \end{equation}
    then
     \[
D^{2} \geq 2\frac{D.F}{h^{0}(F, D_{|_{F}})}\deg f_{*}\mathcal{O}_{S}(D)  \text{.}
\]
First, we remark that the condition \eqref{Condition} is equivalent to 
$$2h^{0}(F, D_{|_{F}}) - D.F - 1 = 1,$$
since by Clifford's Theorem: 
\[
D.F \geq 2(h^{0}(F, D_{|_{F}}) - 1) \text{.}
\]
Thus, we can assume  $D_{|{_F}}$ has a section and is not necessarily generated by global sections since we can always eliminate the horizontal fixed part of $D$. By Corollary \ref{secondmainresult}, we see that whether $D_{|{_F}}$ is special or nonspecial, we proved  the same inequality:
   \[
D^{2} \geq 2\frac{D.F}{h^{0}(F, {D}_{|_{F}})}\deg f_{*}\mathcal{O}_{S}(D)  \text{,}
\]
if the following more general condition holds:
\begin{equation}\label{conditionQ}
D.F = \alpha_{(D, f_{*}\mathcal{O}_{S}(D))} (h^{0}(F, D_{|_{F}}) - 1) \text{.}  \tag{Q}
\end{equation}
\end{example}
\begin{theorem}\label{invarianttheorem}
    let $f:S \to C$ be a fibred surface and $F$ its general fiber. If $D$ is a relatively effective and  relatively nef divisor, $D_{|_{F}}$ is nonspecial with $h^{0}(F, D_{|_{F}}) > g$, then 
    \[
    D^{2} \geq  2\frac{D.F}{h^{0}(F, {D}_{|_{F}})}\deg f_{*}\mathcal{O}_{S}(D)  \text{.}
    \]
    \begin{proof}
        By assumption, $D_{|_{F}}$ is nonspecial and $h^{0}(F, D_{|_{F}}) > g$. Thus, we have $\beta_{D} \leq 2$ and $\alpha_{(D, f_{*}\mathcal{O}_{S}(D))} = \beta_{D}$. Then the condition $(Q)$ is satisfied:
     \[
            D.F = \alpha_{(D, f_{*}\mathcal{O}_{S}(D))}(h^{0}(F, D_{|_{F}}) -1) \text{.}
     \]
     Finally, the desired inequality  follows from  Example \ref{remarkslopepaper}.
    \end{proof}
\end{theorem}
\begin{proposition}\label{proposition pos}
    Let $f : S \to C$ be a fibred surface and $D$ be a nef divisor on $S$. Consider the data $(D, \mathcal{F})$. Then, if $t_{(D, \mathcal{F})} \geq 1$:
    \[
2\epsilon^{*}D.Z_{i} - Z_{i}^{2} \geq 0, \hspace{0.2cm}  \forall i; 1\leq i \leq t_{(D, \mathcal{F})} \text{.}\]     
\end{proposition}
\begin{proof}
    Recall that  $\epsilon^{*}D = M_{i} + Z_{i}$, thus $2\epsilon^{*}D.Z_{i} - Z_{i}^{2} = (\epsilon^{*}D + M_{i})Z_{i}$. However, $M_{i}$ and $\epsilon^{*}D$ are nef, $\forall i; 1\leq i \leq t_{(D,\mathcal{F})}$. So, we deduce the nonnegativity of $2\epsilon^{*}D.Z_{i} - Z_{i}^{2}$.
\end{proof}
\begin{corollary}[{Compare with \cite[Theorem 3.20]{stoppino}}]\label{comparewithstoppino}
    Let $f : S \to C$ be a fibred surface and $D$ be a nef divisor on $S$. Consider the data $(D, \mathcal{F})$ where $\mathcal{F} \subseteq f_{*}\mathcal{O}_{S}(D)$. Then
    \begin{itemize}
        \item[(1).] If $\mathcal{F}$ is semi-stable or $t_{(D,\mathcal{F})} = 1$, then
          \[
D^{2} \geq 2\frac{d_{1}}{r_{1}}\deg \mathcal{F}_{1} \geq 2\frac{d_{1}}{r_{1}}\deg \mathcal{F} \text{.}
\]
 \item[(2).] If $1 < t_{(D,\mathcal{F})} \leq k$, then
  \[
  D^{2} \geq\frac{2\alpha_{(D,\mathcal{F}_{t_{(D,\mathcal{F})}})} d_{t_{(D,\mathcal{F})}}}{d_{t_{(D,\mathcal{F})}} + \alpha_{(D,\mathcal{F}_{t_{(D,\mathcal{F})}})}} \deg \mathcal{F}_{t_{(D,\mathcal{F})}} 
 \geq\frac{2\alpha_{(D,\mathcal{F}_{t_{(D,\mathcal{F})}})} d_{t_{(D,\mathcal{F})}}}{d_{t_{(D,\mathcal{F})}} + \alpha_{(D,\mathcal{F}_{t_{(D,\mathcal{F})}})}} \deg \mathcal{F}
  \text{.} 
  \]  
    \end{itemize}
\end{corollary}
\begin{proof}
    Apply Proposition \ref{proposition pos} and  Theorem \ref{main result}.
\end{proof}
\section{Examples and applications}
\begin{example} Let $D = K_{S/C}$ be the relative canonical divisor of a fibred surface $f \colon S \rightarrow  C$ with $g(F) \geq 2$. Thus, by \cite{Fujita1}, $f_{*}\omega_{S/C}$ is a nef vector bundle on $C$. Also, $\Hat{n}_{(K_{S/C}, f_{*}\omega_{S/C}) } = k$ since $D_{|_{F}} = K_{F}$ is a special divisor on $F$.  This implies $\alpha_{(K_{S/C}, f_{*}\omega_{S/C})} = 2$. We also have $t_{(K_{S/C}, f_{*}\omega_{S/C})} = k$, where $k$ is the length of the Harder-Narasimhan filtration of $f_{*}\omega_{S/C}$, and $d_{k} = 2g-2$ where $g = g(F)$. Then, by Remark \ref{remark4.6} and Theorem \ref{main result}, we have:
\[
K_{S/C}^{2} \geq 4\frac{g-1}{g} \deg f_{*}\omega_{S/C} + 2\epsilon^{*}K_{S/C}.Z_{k} - Z_{k}^{2} \text{.}
\]
Recall that $$2\epsilon^{*}K_{S/C}.Z_{k} - Z_{k}^{2} = 2K_{S/C}.Z_{\mathcal{F}_{k}} - Z_{\mathcal{F}_{k}}^{2} - E^{2} \text{,}$$
where $E$ is the exceptional divisor of $\epsilon$ (Proposition \ref{proposition3.3}). Here $\mathcal{F}_{k} = \mathcal{F} = f_{*}\omega_{S/C}$. In \cite[Example 2.1]{Barjastoppinostability},
the authors considered the case where $f$ is a relatively minimal nodal fibration. They calculate the term $$2K_{S/C}.Z_{\mathcal{F}_{k}} - Z_{\mathcal{F}_{k}}^{2} - E^{2} \text{,}$$ and  prove that  if  $n\geq 1$ is the total number of 
disconnecting nodes contained in the fibres, we have 
\[
2K_{S/C}.Z_{\mathcal{F}_{k}} - Z_{\mathcal{F}_{k}}^{2} - E^{2} \geq n \text{.}
\]
Therefore,
\[
K_{S/C}^{2} \geq 4\frac{g-1}{g} \deg f_{*}\omega_{S/C} + n \text{.}
\]
This implies  if $f$ is relatively minimal and  $$K_{S/C}^{2} = 4\frac{g-1}{g} \deg f_{*}\omega_{S/C} \text{,}$$ then $f$ is never a relatively minimal nodal fibration with at least $1$ disconnected node. 
\par Furthermore, we remark  that if $f$ is relatively minimal, then by Corollary \ref{secondmainresult}, $K_{S/C}$ is nef and therefore $$2K_{S/C}.Z_{\mathcal{F}_{k}} - Z_{\mathcal{F}_{k}}^{2} - E^{2} \geq 0 \text{.}$$ Moreover, to have $$K_{S/C}^{2} = 4\frac{g-1}{g} \deg f_{*}\omega_{S/C} \text{,}$$ it must be that $$2K_{S/C}.Z_{\mathcal{F}_{k}} - Z_{\mathcal{F}_{k}}^{2} - E^{2} = 0\text{.}$$ Then $|K_{S/C} + f^{*}\mathcal{O}(A)|$  has no fixed part and is base point free for  a sufficiently ample divisor on $C$. 
\par In general, if $f$ is relatively minimal, we always have the original Xiao's result  \cite{xiaogang}:
\[
K_{S/C}^{2} \geq 4\frac{g-1}{g} \deg f_{*}\omega_{S/C} \text{.}
\]
\end{example}
\begin{example}\label{pluricanonicalexample}
Let $D = mK_{S/C}$ be the relative  pluricanonical divisor of a fibred surface $f \colon S \rightarrow C$, with $g=g(F) \geq 2$, and $m \geq 2$. It is well known that $f_{*}\omega_{S/C}^{\otimes m}$ is a nef vector bundle on $C$, see \cite[Theorem 1.3]{eckartpositivity} for instance. If $K_{S/C}$ is relatively nef, then by Corollary \ref{secondmainresult}, we have the following cases:
\begin{itemize}
    \item If $f_{*}\omega_{S/C}^{\otimes m}$ is semi-stable:
    \[
    m^{2}K_{S/C}^{2} \geq  \frac{4m}{2m-1} \deg f_{*}\omega_{S/C}^{\otimes m}
    \text{,}\]
    since $$\rk(f_{*}\omega_{S/C}^{\otimes m}) = (2m-1)(g-1)\text{.}$$ 
    Therefore,
      \[
    K_{S/C}^{2} \geq  \frac{4}{m(2m-1)} \deg f_{*}\omega_{S/C}^{\otimes m}
   \text{.} \]
    \item  If $f_{*}\omega_{S/C}^{\otimes m}$ is not semi-stable, then using the fact that $mK_{F}$ is a nonspecial divisor, we see that 
    \[
    \Hat{n}_{(mK_{S/C}, f_{*}\omega_{S/C}^{\otimes m})} < k \text{,}
    \]
and
\[\alpha_{(mK_{S/C}, f_{*}\omega_{S/C}^{\otimes m})} = 1 + \frac{g}{h^{0}(F, mK_{F})-1} \text{,}\]  
    since $$h^{0}(F, mK_{F}) > g \text{.}$$ 
    Moreover,
    \[
    t_{(mK_{S/C}, f_{*}\omega_{S/C}^{\otimes m})} = k \text{,}    \]
    because $f_{*}\omega_{S/C}^{\otimes m}$ is a nef vector bundle on $C$. Then,
    \[
    m^{2}K_{S/C} \geq \frac{2\alpha_{(mK_{S/C}, f_{*}\omega_{S/C}^{\otimes m})} d_{k}}{d_{k} + \alpha_{(mK_{S/C}, f_{*}\omega_{S/C}^{\otimes m})}} 
  \deg f_{*}\omega_{S/C}^{\otimes m}\text{.}\]
  Recall that $d_{k} = 2m(g-1)$ and $h^{0}(F, mK_{F}) = (2m-1)(g-1)$, thus
  \[
  \frac{2\alpha_{(mK_{S/C}, f_{*}\omega_{S/C}^{\otimes m})} d_{k}}{m^{2}(d_{k} + \alpha_{(mK_{S/C}, f_{*}\omega_{S/C}^{\otimes m})})}= \frac{4}{m(2m-1)} \text{.}
  \]
  Therefore,
  \[ K_{S/C}^{2} \geq \frac{4}{m(2m-1)} \deg f_{*}\omega_{S/C}^{\otimes m}
\text{.}  \]
\end{itemize}
In this computations, we see that whether $f_{*}\omega_{S/C}^{\otimes m}$ is semi-stable or not, we have the same lower bound. The reason is that $$d_{k} = \alpha_{(mK_{S/C}, f_{*}\omega_{S/C}^{\otimes m} )}.(r_{k} - 1) \text{,}$$ as explained in the last item of  Corollary \ref{secondmainresult}. However,  in general, for  data $(D, \mathcal{F})$ where $\mathcal{F} \subsetneq f_{*}\mathcal{O}_{S}(D)$, the constant $\frac{2\alpha_{(D,\mathcal{F})}d_{k}}{d_{k} + \alpha_{(D, \mathcal{F})}}$ is different from $\frac{2d_{k}}{r_{k}}$ because the inequality $d_{k} > \alpha_{(D, \mathcal{F})}(r_{k} -1)$ can  well happen.
\end{example}
\begin{example}\label{relativeexample}
    Now, let $f: S \to C$ be a fibred surface and $F$ its general fiber with $g(F) \geq 2$. We take $D = K_{S/C} + L$ such that $L$ is nef and relatively big with $L.F >1$. We know that $f_{*}\mathcal{O}_{S}(D)$ is a nef vector bundle on $C$ with  $$\rk(f_{*}\mathcal{O}_{S}(D)) = g - 1 + L.F \neq 0 \text{.}$$ If $D$ is a relatively nef divisor, then by Corollary \ref{secondmainresult}, we have the following lower bound for $(K_{S/C} + L)^{2}$:
        \[
        (K_{S/C} + L)^{2} \geq 2\left(1 + \frac{g-1}{g - 1 + L.F}\right)\deg f_{*}(\omega_{S/C} \otimes \mathcal{O}(L) )    \text{,}\]
        because we know the following information for the data $(D, f_{*}\mathcal{O}_{S}(D))$: 
        $$d_{k} = \alpha_{(D,f_{*}\mathcal{O}_{S}(D))}(r_{k} -1), \hspace{0.2cm} \Hat{n}_{(D, f_{*}\mathcal{O}_{S}(D))} < k \text{,} \hspace{0.2cm}
        \alpha_{(D,f_{*}\mathcal{O}_{S}(D))} = 1 + \frac{g}{g-2+L.F} \text{,}
        $$
also the quantity $t_{(D,f_{*}\mathcal{O}_{S}(D))}$ is maximal, this means $t_{(D, f_{*}\mathcal{O}_{S}(D))} = k \text{.}$ 
\end{example}
\begin{example}\label{exampleusingfujitadecomposition}
We give a trivial example in which we can see that $t_{(D, f_{*}\mathcal{O}_{S}(D))} = -\infty$ can happen. 
Let $f:S \to C$ be a fibred surface and $F$ its general fiber with $g(F) = 2$. Assume that $K_{S/C}$ is a nef divisor and the second Fujita decomposition (Theorem \ref{Second:Fujita:decomp:thm}) of $f_{*}\omega_{S/C}$ is not trivial. Then there exists an ample line bundle $\mathcal{A}$ on $C$ and a flat line bundle $\mathcal{U}$ such that
\[
f_{*}\omega_{S/C} = \mathcal{A} \oplus \mathcal{U} \text{.}
\]
Now, let $D = K_{S/C} + f^{*}\mathcal{M}$
for a  sufficiently negative divisor $\mathcal{M}$ on $C$ such that $\deg(\mathcal{A} \otimes \mathcal{O}(\mathcal{M})) < 0$. 
In this case, $D$ is a  relatively nef divisor on $S$ such that $D_{|_{F}} = K_{F}$ and the bundle $f_{*}\mathcal{O}_{S}(D)$ decomposes into two parts: 
\[
f_{*}\mathcal{O}_{S}(D) = \mathcal{A}\otimes \mathcal{M} \oplus \mathcal{U} \otimes \mathcal{M} \text{.}
\] 
The Harder-Narasimhan filtration of $f_{*}\mathcal{O}_{S}(D)$ is the following: 
\[
0 \subsetneq \mathcal{A} \otimes \mathcal{M} \subsetneq f_{*}\mathcal{O}_{S}(D) \text{.}
\]
We consider the data $(D, f_{*}\mathcal{O}_{S}(D))$. Since $D_{|_{F}} =  K_{F}$, we see that  $D_{|_{F}}$ is special,  $\Hat{n}_{(D, f_{*}\mathcal{O}_{S}(D))} = 2$, and $\alpha_{(D,f_{*}\mathcal{O}_{S}(D))} =2$. Since $g(F) = 2$, then we have $d_{2}= 2$. Also, clearly, we have $t_{(D, f_{*}\mathcal{O}_{S}(D))} = -\infty$. Applying Corollary \ref{secondmainresult}, we deduce the following inequality: 
\[
D^{2} \geq C_{(D, f_{*}\mathcal{O}_{S}(D))} \deg f_{*}\mathcal{O}_{S}(D) \text{.}
\]
We know that  $C_{(D, f_{*}\mathcal{O}_{S}(D))} = 2$, so we replace $C_{(D, f_{*}\mathcal{O}_{S}(D))}$ by $2$ in the above inequality:
\[
D^{2} \geq 2 \deg f_{*}\mathcal{O}_{S}(D) \text{.}
\]
Additionally, we remark that $d_{k} = \alpha_{(D, f_{*}\mathcal{O}_{S}(D))}(r_{k} -1)$. Thus, again by Corollary \ref{secondmainresult} or Example \ref{remarkslopepaper}, we deduce the same lower bound. 
\end{example}
\begin{proposition}\label{pluricaonicalproposition}
    Let $f : S \to C$ be a fibred surface and $F$ its general fiber with $g(F) \geq 2$. Then
    \[
    K_{S/C}^{2} = \frac{2}{m(m-1)} \left( \deg f_{*}\omega_{S/C}^{\otimes m} - \deg f_{*}\omega_{S/C}\right), \hspace{0.2cm} \forall m\geq 2 \text{.}
    \]
\end{proposition}
\begin{remark}
In the setting of Proposition \ref{pluricaonicalproposition}, in particular: 
      \[
    K_{S/C}^{2} = \deg f_{*}\omega_{S/C}^{\otimes 2} - \deg f_{*}\omega_{S/C}  \text{.}
    \]
\end{remark}
\begin{proof}
Recall that if $S \to C$ is a fibred surface, then for any line bundle $\mathcal{L} = \mathcal{O}_{S}(L)$ on $S$, we have the following Grothendieck-Riemann-Roch formula \cite[page 333]{ACGH}: 
\[
\deg f_{!}\mathcal{L} = \deg f_{*}\mathcal{L} - \deg \mathcal{R}^{1}f_{*}\mathcal{L} = \frac{L(L-K_{S/C})}{2} + \deg f_{*}\omega_{S/C} \text{.}
\]
In particular, we apply the above formula to the relative pluricanonical bundle $\omega_{S/C}^{\otimes m}$, thus we obtain
\[
\deg f_{*}\omega_{S/C}^{\otimes m} - \deg \mathcal{R}^{1}f_{*}\omega_{S/C}^{\otimes m} = \frac{m(m-1)K_{S/C}^{2}}{2} + \deg f_{*}\omega_{S/C} \text{.}
\]
But  $\mathcal{R}^{1}f_{*}\omega_{S/C}^{\otimes m} = 0$ because $\mathcal{R}^{1}f_{*}\omega_{S/C}^{\otimes m}$ is known to be torsion free and
\[
\rk \mathcal{R}^{1}f_{*}\omega_{S/C}^{\otimes m} = h^{1}(F, \omega_{F}^{\otimes m}) = h^{0}(F, \omega_{F}^{\otimes (1-m)}) = 0,
\]
since $g(F) \geq 2$ by assumption. Thus, we obtain the  desired formula: 
     \[
    K_{S/C}^{2} = \frac{2}{m(m-1)} \left( \deg f_{*}\omega_{S/C}^{\otimes m} - \deg f_{*}\omega_{S/C}\right), \hspace{0.2cm} \forall m\geq 2 \text{.}
    \]
\end{proof}
\begin{remark}\label{xiaoremark}
    Using  Noether's Formula, Xiao \cite[Theorem 2]{xiaogang} remarked that
    \[
    K_{S/C}^{2} \leq 12 \deg f_{*}\omega_{S/C} \text{.}
    \]
\end{remark}
\begin{proposition} \label{proposition5.7}
    Let $f : S \to C$ be a fibred surface and $F$ its general fiber with $g(F) \geq 2$. Then
    \[
    K_{S/C}^{2} \leq \frac{12}{6m(m-1) + 1}  \deg f_{*}\omega_{S/C}^{\otimes m}, \hspace{0.2cm} \forall m\geq 2 \text{.}
    \]
\end{proposition}
\begin{proof}
We recall that in Proposition \ref{pluricaonicalproposition}, we proved that 
\[
    K_{S/C}^{2} = \frac{2}{m(m-1)} \left( \deg f_{*}\omega_{S/C}^{\otimes m} - \deg f_{*}\omega_{S/C}\right), \hspace{0.2cm} \forall m\geq 2
    \]
    which implies
    \[
     \frac{m(m-1)}{2}K_{S/C}^{2} = \deg f_{*}\omega_{S/C}^{\otimes m} - \deg f_{*}\omega_{S/C}, \hspace{0.2cm} \forall m\geq 2\text{.}   \]
    Thus,
\begin{equation}\label{equality5.1} 
\deg f_{*}\omega_{S/C} = \deg f_{*}\omega_{S/C}^{\otimes m} - \frac{m(m-1)}{2}K_{S/C}^{2}, \hspace{0.2cm} \forall m\geq 2 \text{.}
\end{equation}
Now, by Remark \ref{xiaoremark}, we have
  \[
    K_{S/C}^{2} \leq 12 \deg f_{*}\omega_{S/C} \text{.}
    \]
Combining this inequality with equality (\ref{equality5.1}), it follows that
\[
K_{S/C}^{2} \leq 12 \left(\deg f_{*}\omega_{S/C}^{\otimes m} - \frac{m(m-1)}{2}K_{S/C}^{2}\right), \hspace{0.2cm} \forall m\geq 2
\text{,}\]
which implies
\[ 
K_{S/C}^{2} \leq 12 \deg f_{*}\omega_{S/C}^{\otimes m} - 6m(m-1)K_{S/C}^{2}, \hspace{0.2cm} \forall m\geq 2 \text{.}
\]
Therefore,
\[
(6m(m-1)+1)K_{S/C}^{2} \leq 12 \deg f_{*}\omega_{S/C}^{\otimes m}, \hspace{0.2cm} \forall m\geq 2 \text{,}
\]
which further simplifies to the desired inequality:
\[
K_{S/C}^{2} \leq \frac{12}{6m(m-1)+1}\deg f_{*}\omega_{S/C}^{\otimes m}, \hspace{0.2cm} \forall m\geq 2 \text{.}
\]
\end{proof}
\begin{remark}
    If $K_{S/C}^{2} > 0$ and $g(F) \geq 2$, then  $\deg f_{*}\omega_{S/C}^{\otimes m} > 0$, $\forall m \geq 2$ \text{.}
\end{remark}
\begin{lemma}\label{Lemma 5.8}
    Let $f : S \to C$ be a fibred surface, $F$ its general fiber with $g(F) \geq 2$, and  $\mathcal{F} \subseteq f_{*}\omega_{S/C}^{\otimes m}$ with $\deg \mathcal{F} \geq 0$, $m \geq 2$. Then
    \[
    \mu(\mathcal{F}) \leq \frac{6m}{(6m(m-1) + 1)(g-1)} \deg f_{*}\omega_{S/C}^{\otimes m}\text{,}
     \]
      where $\mu(\mathcal{F}) = \frac{\deg \mathcal{F}}{\rk \mathcal{F}}$.
\end{lemma}

\begin{proof}
    We consider the data $(mK_{S/C}, \mathcal{F})$ and let $\mathcal{F^{'}}$ be the maximal destabilizing vector sub-bundle of $\mathcal{F}$. We apply  Lemma \ref{4.1}
    to the sequences  $\{Z(\mathcal{F}^{'}), 0\}$ and $(\mu(\mathcal{F}^{'}), 0)$. Then, we have
    \[
    K_{S/C}^{2} \geq \frac{2}{m}(g-1)\mu(\mathcal{F}^{'})
    \text{.}\]
    Here, $Z(\mathcal{F}^{'})$ is the fixed part of the sub-bundle $\mathcal{F^{'}}$ and $\mu(\mathcal{F}^{'}) = \frac{\deg \mathcal{F}^{'}} {\rk \mathcal{F}^{'}}$. Since $\mathcal{F}^{'}$ is the maximal destabilizing vector sub-bundle of $\mathcal{F}$, we have
    \[
\mu(\mathcal{F}^{'}) \geq \mu(\mathcal{F}) \text{.}    \]
Combining the last two inequalities and Proposition \ref{proposition5.7}, we deduce the desired inequality:
      \[
    \mu(\mathcal{F}) \leq \frac{6m}{(6m(m-1) + 1)(g-1)} \deg f_{*}\omega_{S/C}^{\otimes m}
     \text{.}\]
\end{proof}
Now, we state the First Fujita decomposition for the relative pluricanonical bundle and adjoint canonical bundle in the case of a fibred surface. 
\begin{theorem} \label{firstfujita}
    Let $f: S \to C$ be a fibred surface and $F$ its general fiber, let $L$ be a semi-ample line bundle on $S$. Then
    \[
f_{*}\omega_{S/C}^{\otimes m} = \mathcal{N}_{m} \oplus \mathcal{O}_{C}^{\oplus p_{m}}, \hspace{0,2cm} \forall m \geq 2\text{,} \text{ and }
f_{*}(\omega_{S/C} \otimes L) = \mathcal{N}_{L} \oplus \mathcal{O}_{C}^{\oplus p_{L}} \text{.}
\]
Here, $H^{0}(C, \mathcal{N}_{m}^{\vee}) = 0$ and $H^{0}(C, \mathcal{N}_{L}^{\vee}) = 0 \text{.}$
\end{theorem}
\begin{proof}
Note $p_{m}:= h^{0}(C, (f_{*}\omega_{S/C}^{\otimes m}) ^{\vee}) = h^{0}(C, \mathcal{R}^{1}f_{*}\omega_{S/C}^{\otimes(1-m)})$,  since $(f_{*}\omega_{S/C}^{\otimes m}) ^{\vee} \simeq \mathcal{R}^{1}f_{*}\omega_{S/C}^{\otimes(1-m)} $. We take $\{s_{1}, \dots, s_{p_{m}}\}$ as a basis of $H^{0}(C, \mathcal{R}^{1}f_{*}\omega_{S/C}^{\otimes(1-m)}) \simeq Hom(\mathcal{O}_{C}, \mathcal{R}^{1}f_{*}\omega_{S/C}^{\otimes(1-m)})$. Then $s_{1} \oplus \dots \oplus s_{p_{m}}$ defines a map:
\[
s_{1} \oplus \dots \oplus s_{p_{m}}: \mathcal{O}_{C}^{\oplus p_{m}} \longrightarrow \mathcal{R}^{1}f_{*}\omega_{S/C}^{\otimes(1-m)}
\]
which yields the following short exact sequence: 
\[
0 \longrightarrow \mathcal{O}_{C}^{\oplus p_{m}} \longrightarrow \mathcal{R}^{1}f_{*}\omega_{S/C}^{\otimes(1-m)} \longrightarrow \mathcal{Q}_{m} \longrightarrow 0
\]
where $\mathcal{Q}_{m}$ is the quotient bundle. By duality, we have
\[
0 \longrightarrow \mathcal{Q}_{m}^{\vee} \longrightarrow f_{*}\omega_{S/C}^{\otimes m} \longrightarrow \mathcal{O}_{C}^{\oplus p_{m}} \longrightarrow 0 \text{.}
\]
We deduce from \cite[Theorem 26.4]{HSP} that the last exact sequence splits since  the bundle $f_{*}\omega_{S/C}^{\otimes m}$ admits a singular hermitian metric with semi-positive curvature and verifies the minimal extension property, see \cite{HSP} for more details. Hence, the first Fujita decomposition follows:
\[
f_{*}\omega_{S/C}^{\otimes m} = \mathcal{N}_{m} \oplus \mathcal{O}_{C}^{\oplus p_{m}}
\]
Here, $\mathcal{N}_{m} = \mathcal{Q}_{m}^{\vee}$. and $h^{0}(C, \mathcal{Q}_{m}) = 0$.
\par For the adjoint case, we note $p_{L}:= h^{0}(C, (f_{*}(\omega_{S/C} \otimes L)) ^{\vee})$, the proof is the same as in the relative pluricanonical case, since $L$ admits a smooth hermitian metric $h$ with semi-positive curvature and a trivial multiplier ideal sheaf. Thus, $f_{*}(\omega_{S/C} \otimes L)$ admits a singular hermitian metric with semi-positive curvature and verifies the minimal extension property. 
\end{proof}
In \cite{schnell}, the authors proved  that the vector bundles $f_{*}\omega_{S/C}^{\otimes m}$ and $f_{*}(\omega_{S/C} \otimes L)$, where $L$ is a semi-ample line bundle on $S$, admit  
a Catanese-Fujita-Kawamata decomposition, $ \forall m \geq 2$. We state the result in the case of fibred surfaces. 
\begin{theorem}[{\cite[Theorem 2]{schnell}}]\label{secondfujita}
    Let $f: S \to C$ be a fibred surface and $F$ its general fiber. Let  $L$ be a semi-ample line bundle on $S$. Then,
    \[
    f_{*}\omega_{S/C}^{\otimes m} = \mathcal{A}_{m} \oplus \mathcal{U}_{m}, \hspace{0,2cm} \forall m \geq 2\text{,} \text{ and }
    f_{*}(\omega_{S/C} \otimes L) = 
    \mathcal{A}_{L} \oplus \mathcal{U}_{L} \text{.}
    \]
    
    Here, $\mathcal{A}_{m}$ and $\mathcal{A}_{L}$ are  ample vector sub-bundles of $f_{*}\omega_{S/C}^{\otimes m}$ and  $f_{*}(\omega_{S/C} \otimes L)$ respectively,  $\mathcal{U}_{m}$ and $\mathcal{U}_{L}$ are hermitian  flat vector sub-bundles of $f_{*}\omega_{S/C}^{\otimes m}$ and $f_{*}(\omega_{S/C} \otimes L)$ respectively.
\end{theorem}
\par In the following paragraphs,  we  derive some explicit consequences for the direct image of relative pluricanonical bundles.
\begin{example}
If $f$ is not isotrivial, it is known that   $f_{*}\omega_{S/C}^{\otimes m}$ is ample if not zero  $ \forall m \geq 2$. This implies that
the trivial part in the First Fujita decomposition is zero, and 
the flat part in the Catanese-Fujita-Kawamata decomposition is zero. 
\end{example}
\begin{example}
If $L$ is an ample line bundle on $S$, it is  well known that $f_{*}(\omega_{S/C} \otimes L)$ is an ample vector bundle on $C$. Equivalently, the trivial part in the First Fujita decomposition is zero, and 
the flat part in the Catanese-Fujita-Kawamata decomposition is zero.
\end{example}
\par Recall  an important result on the direct image of pluricanonical sheaf over a curve due to Viehweg \cite[Theorem 1.3]{eckartpositivity} and \cite[Proposition 4.6]{eckartpositivity}. We restrict ourselves to a fibred surface case. 
\begin{proposition}\label{viehwegproposition}
    Let $f: S \to C$ be a fibred surface and $F$ be its general fiber. The following conditions are equivalent:
    \begin{itemize}
        \item[(1).] For all $m \geq 2$, the vector bundle $f_{*}\omega_{S/C}^{\otimes m}$ is ample if not zero.
        \item[(2).] There exist some $m \geq 1$ such that $f_{*}\omega_{S/C}^{\otimes m}$ contains an ample sub-sheaf.
        \item[(3).]  There exist some $m \geq 1$ such that $ \deg f_{*}\omega_{S/C}^{\otimes m} > 0$.  
    \end{itemize}
Moreover, if $f$ is semi-stable, then the conditions $(1), (2)$, and $(3)$ are equivalent to
    \begin{itemize}
        \item[(4).] $f$ is not isotrivial. 
    \end{itemize}
 \end{proposition}
\begin{corollary}[Compare with {\cite[Corollary 5.2]{schnell}}]\label{schnell result}
    Let $f: S \to C$ be a fibred surface. Then, either  $f_{*}\omega_{S/C}^{\otimes m}$ is ample, $\forall m \geq 2$ when $f_{*}\omega_{S/C}^{\otimes m} \neq 0$, or $f_{*}\omega_{S/C}^{\otimes m}$ is hermitian flat $\forall m \geq 2$. 
\end{corollary} 
\begin{proof}
    Apply  Proposition \ref{viehwegproposition} and  Theorem \ref{secondfujita}.
\end{proof}
\begin{remark}
    If $f$ is semi-stable, then $f_{*}\omega_{S/C}^{\otimes m}$ being hermitian flat $\forall m \geq 2$ is equivalent to $f$ being isotrivial.
\end{remark}
\begin{remark}
By  Corollary \ref{schnell result}, we observe that $f_{*}\omega_{S/C}^{\otimes m}$ is ample $\forall m \geq 2$ when $f_{*}\omega_{S/C}^{\otimes m} \neq 0$ or $f_{*}\omega_{S/C}^{\otimes m}$ is hermitian flat $\forall m \geq 2$. According to Proposition \ref{viehwegproposition}, if $f_{*}\omega_{S/C}^{\otimes m}$ is flat $\forall m \geq 2$, then $f_{*}\omega_{S/C}$ is flat. Conversely, if $f_{*}\omega_{S/C}^{\otimes m}$ is ample  $\forall m \geq 2$, $g(F) \geq 2$, and $K_{S/C}$ is nef, then $f_{*}\omega_{S/C}$ is not flat by  Example \ref{pluricanonicalexample}.
\end{remark}
In this last paragraph, we will explore the relationship between $\deg f_{*}\omega_{S/C}$  and $\deg f_{*}\omega_{S/C}^{\otimes m}$  in the case where $f_{*}\omega_{S/C}^{\otimes m}$ is ample $\forall m \geq 2$.
\begin{corollary}
    Let $f:S \to C$ be a fibred surface and $F$ its general fiber with $g(F) \geq 2$. Then:
    \[
    \deg f_{*}\omega_{S/C} \leq \frac{6mg}{(6m(m-1) + 1)(g-1)} \deg f_{*}\omega_{S/C}^{\otimes m}\text{,}
\hspace{0.2cm} \forall m \geq 2   \text{.}\]
In particular, if $g > 6m(m-1) + 1$, then:
\[
    \deg f_{*}\omega_{S/C} < \frac{1}{m-1} \deg f_{*}\omega_{S/C}^{\otimes m}\text{,}
\hspace{0.2cm} \forall m \geq 2   \text{.}\]
\end{corollary}
\begin{proof}
    Apply  Lemma \ref{Lemma 5.8}.
\end{proof}
\begin{example}
If $m=2$ and $g > 13$, then $\frac{12g}{13(g-1)} < 1$, and
    \[
    \deg f_{*}\omega_{S/C} \leq \frac{12g}{13(g-1)} \deg f_{*}\omega_{S/C}^{\otimes 2}\text{.} \]
    If $m =3$, then 
       \[
    \deg f_{*}\omega_{S/C} \leq \frac{18g}{37(g-1)} \deg f_{*}\omega_{S/C}^{\otimes 3}\text{.} \]
In particular, if $g > 37$, we deduce that  $\frac{18g}{37(g-1)} < \frac{1}{2}$
    and 
    \[
    \deg f_{*}\omega_{S/C} < \frac{1}{2} \deg f_{*}\omega_{S/C}^{\otimes 3}\text{.} \]
     If $m = 4$, then
     \[
    \deg f_{*}\omega_{S/C} \leq \frac{24g}{73(g-1)} \deg f_{*}\omega_{S/C}^{\otimes 4}\text{.} \]
Furthermore, if $g > 73$, then $\frac{24g}{73(g-1)} < \frac{1}{3}$
    and 
    \[
    \deg f_{*}\omega_{S/C} < \frac{1}{3} \deg f_{*}\omega_{S/C}^{\otimes 4}\text{.} \] 
    \end{example}

\end{document}